\documentclass[12pt]{amsart}
\usepackage{amssymb,amsmath,amsthm,MnSymbol,arydshln}
\usepackage{graphicx}
\usepackage{tikz-cd}
\usepackage{caption}   
\usepackage{pgfplots}
\pgfplotsset{compat=1.18} 
\usepackage{hyperref}
\usepackage[msc-links]{amsrefs}
\hypersetup{
  colorlinks   = true,	
  urlcolor     = blue,
  linkcolor    = blue,
  citecolor   = red
}
\calclayout
\allowdisplaybreaks
\newtheorem{theorem}{Theorem}

\newtheorem{remark}[theorem]{Remark}
\newtheorem{corollary}[theorem]{Corollary}

\numberwithin{theorem}{section} \numberwithin{equation}{section}
\newcommand{\beq}{\begin{small} \begin{equation}}
\newcommand{\eeq}{\end{equation} \end{small}}
\newcommand{\beqn}{\begin{small} \begin{equation*}}
\newcommand{\eeqn}{\end{equation*} \end{small}}
\DeclareMathAlphabet{\mathpzc}{OT1}{pzc}{m}{it}
\newcommand\scalemath[2]{\scalebox{#1}{\mbox{\ensuremath{\displaystyle #2}}}}
\begin{document}
\title{On N\'eron-Severi lattices of Jacobian elliptic K3 surfaces}
\begin{abstract}
We classify all Jacobian elliptic fibrations on K3 surfaces with finite automorphism group. We also classify all Jacobian elliptic fibrations with finite Mordell-Weil group on K3 surfaces with infinite automorphism group and 2-elementary N\'eron-Severi lattice.  As part of the classification, we compute the lattice theoretic multiplicities of all Jacobian elliptic fibrations in both cases.
\end{abstract}
\author{Adrian Clingher}
\address{Department of Mathematics and Statistics, University of Missouri - St. Louis, St. Louis, MO 63121}
\email{clinghera@umsl.edu}
\author{Andreas Malmendier}
\address{Department of Mathematics \& Statistics, Utah State University, Logan, UT 84322}
\email{andreas.malmendier@usu.edu}
\maketitle
\section{Introduction}
The goal of this note is as follows. First, we provide a classification for Jacobian elliptic fibrations on complex algebraic K3 surfaces with finite automorphism group. Namely, 
we list all lattices $L$ that may occur as N\'eron-Severi lattice for a K3 surface $\mathcal{X}$ with finite automorphism group and then, for each possible $L$, we determine all possible Jacobian elliptic fibrations supported on $\mathcal{X}$, in terms of reducible fiber types and Mordell-Weil groups. Second, we consider the case of K3  surfaces $\mathcal{X}$ with 2-elementary N\'eron-Severi lattice $L$ and infinite automorphism group. In this case, we classify all Jacobian elliptic fibrations with finite Mordell-Weil group in terms of reducible fiber types and actual Mordell-Weil groups. The latter group includes interesting K3 examples, such as double-sextic K3 surfaces, twisted Legendre pencils, and Kummer surfaces associated with products of two elliptic curves. The former group also includes interesting objects, such as the Shioda-Inose K3 surfaces \cite{MR728142}. Finally, we identify the Jacobian elliptic fibrations in both cases where the lattice embedding of the rank-2 hyperbolic lattice spanned by the cohomology classes associated with the elliptic fiber and the section into the polarizing lattice is unique. The existence of \emph{unique} Jacobian elliptic fibrations is crucial when proving unirationality for certain moduli spaces of K3 surfaces \cite{CM:2022}.
\par The classification results are obtained via lattice theoretic methods and are based, in essence, on a combination two classical results. The first is  Nikulin's classification \cite{MR633160b}  of hyperbolic, even, 2-elementary lattices $L$ admitting a primitive embedding into the K3 lattice $\Lambda_{K3}\cong H^{\oplus 3} \oplus E_8(-1)^{\oplus 2}$. The second is Shimada's general result \cite{MR1813537} classifying Jacobian elliptic fibrations on K3 surfaces.  A classification of  elliptic fibrations on K3 surfaces with 2-elementary N\'eron-Severi lattice and finite automorphism group was given in \cite{MR4130832}. Elliptic fibrations in the case with 2-elementary N\'eron-Severi lattice and infinite automorphism group were constructed in \cite{MR3201823}.  Some other surfaces relevant to our classification are $p$-elementary for an odd prime number $p$. If the N\'eron-Severi group of a K3 surface is $p$-elementary, then the K3 surface admits a non-symplectic automorphism of order $p$. Their elliptic fibrations were intensively studied in \cites{MR2805445, MR2443767, MR3009163}. Our result will provide a classification of all Jacobian elliptic fibrations on K3 surfaces with finite automorphism group in Theorem~\ref{prop1}.  In the case of an infinite automorphism group, we classify all Jacobian elliptic fibrations with finite Mordell-Weil group on K3 surfaces with 2-elementary N\'eron-Severi lattice in Theorem~\ref{prop2}.  Finally, in Theorem~\ref{prop3} we identify the Jacobian elliptic fibrations in both cases whose lattice theoretic multiplicities equal one. For the latter, we compute the multiplicities of all Jacobian elliptic fibrations from Theorems~\ref{prop1} and~\ref{prop2} using {\sc Sage} and {\sc Magma}. As we will show, for certain root lattices with Mordell-Weil groups $\mathbb{Z}/2\mathbb{Z}$ or $(\mathbb{Z}/2\mathbb{Z})^2$ there exist distinct overlattices that can even give rise to distinct lattice polarizations (differing by their parity). The computed multiplicity differentiates these cases effectively.
\section*{Acknowledgement}
The authors would like to thank Don Taylor,  Markus Kirschmer, and Simon Brandhorst for providing valuable help with the {\sc Magma} and {\sc Sage} programs we used for our computations. The authors also thank Alice Garbagnati and Cec\'ilia Salgado, Gabriele Nebe, John Cannon, and Alex Ghitza for helpful discussions. Furthermore, the authors thank the referees for their comments and corrections. A.C. acknowledges support from a UMSL Mid-Career Research Grant. A.M. acknowledges support from the Simons Foundation through grant no.~202367.
\section{Notation and Classical Results}
Let us start by reviewing a few classical lattice theory facts. We shall use the following standard notations: $L_1 \oplus L_2$ is orthogonal direct sum of the two lattices $L_1$ and $L_2$, $L(\lambda)$ is obtained from the lattice $L$ by multiplication of its form by $\lambda \in \mathbb{Z}$, $\langle R \rangle$ is a lattice with the matrix $R$ in some basis; $A_n$, $D_m$, and $E_k$ are the positive definite root lattices for the corresponding root systems;  $H$ is the unique even unimodular hyperbolic rank-two lattice, and $N$ is the negative definite rank-eight Nikulin lattice (as defined, for instance, in \cite{MR728142}*{Sec.~5}).  
\par Given a lattice $L$, one has the \emph{discriminant group} $ = L^\vee/L$ and its associated discriminant form, denoted $q_L$.   A lattice $L$ is then called \emph{2-elementary} if $D(L)$ is a 2-elementary abelian group, i.e., $D(L) \cong (\mathbb{Z}/2\mathbb{Z})^\ell$ where $\ell$ is the {\it length} of $L$,  the minimal number of generators of the group $D(L)$.  One also has the {\it parity} $\delta \in \{ 0,1\}$. By definition, $\delta = 0$ if $q_L(x)$ takes values in $\mathbb{Z}/2\mathbb{Z} \subset \mathbb{Q}/2\mathbb{Z}$ for all $x \in D(L)$, and $ \delta=1$ otherwise. In this context, a result of  Nikulin \cite{MR633160b}*{Thm.~4.3.2} asserts that hyperbolic, even, 2-elementary lattices admitting an embedding into the K3 lattice are uniquely determined by their rank $\rho$, length $\ell$, and parity $\delta$. For instance, $H(2)$, $D_{4n}$ (for $n \in \mathbb{N}$)  are 2-elementary lattices with $\ell=2$, $\delta=0$, while the lattices $A_1$, $E_7$, $\langle \pm 2\rangle$ are 2-elementary with $\ell=1$, $\delta=1$.
\par Based on the above, Nikulin  \cite{MR633160b}*{Tab.~1} provides a list of all hyperbolic, even, 2-elementary lattices $L$ admitting a primitive embedding in $\Lambda_{K3}$ and, hence, occurring as N\'eron-Severi lattice of a K3 surface $\mathcal{X}$. The graph is shown in Figure~\ref{fig:K3_2elet}. 
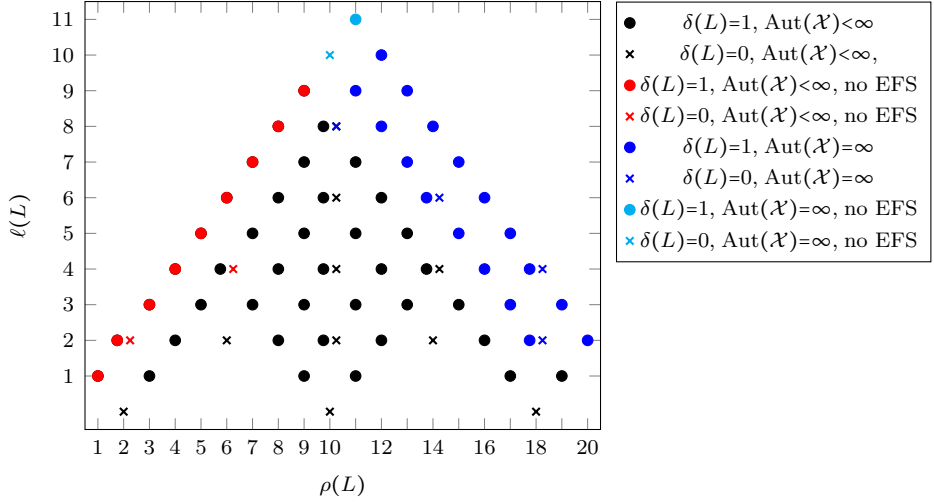
\begin{figure}[!ht]
\begin{center}
\begin{tikzpicture}[scale=1]
	\begin{axis}[
		xlabel=$\scriptstyle \rho(L)$,
		ylabel=$\scriptstyle \ell(L)$,
		tick label style={font=\tiny},
		xmin=0.5, xmax=20.5,
		xtick={1,2,3,4,5,6,7,8,9,10,11,12,13,14,15,16,17,18,19,20}, xticklabels={1,2,3,4,5,6,7,8,9,10,,12,,14,,16,,18,,20},
		ymin=-0.5, ymax=11.5,
		ytick={1,2,3,4,5,6,7,8,9,10,11}, yticklabels={1,2,3,4,5,6,7,8,9,10,11},
		legend style={legend pos=outer north east}
		]
	\addplot[only marks, mark=*, color=black]
		coordinates{ 
      			(1,1)
      			(1.75,2)
      			(3,1) 		(3,3)
      			(4,2) 		(4,4)
      			(5,3) 		(5,5)
      			(5.75,4)	(6,6)
      			(7,3) 		(7,5) 		(7,7)
      			(8,2) 		(8,4) 		(8,6) 		(8,8)
      			(9,1)  	(9,3)		(9,5)  	(9,7)		(9,9)
      			(9.75,2) 	(9.75,4) 	(9.75,6) 	(9.75,8)
      			(11,1) 	(11,3) 	(11,5) 	(11,7)
     			(12,2) 	(12,4) 	(12,6) 	
      			(13,3)  	(13,5)  	
      			(13.75,4)
     			(15,3) 
      			(16,2) 		
      			(17,1) 
     			(19,1) 
    		};
	\addlegendentry{$\scriptstyle  \delta(L)=1, \; \mathrm{Aut}(\mathcal{X})<\infty$};
    	\addplot[only marks, mark=x, thick, color=black]
    		coordinates{
      			(2,0) 		
      			(6,2) 
     			(10,0) 	(10.25,2)		(10.25,4) 		(10.25,6) 		(10.25,8)
      			(14,2)  	(14.25,4) 
      			(18,0) 
		};
     	\addlegendentry{$\scriptstyle  \delta(L)=0, \; \mathrm{Aut}(\mathcal{X})<\infty,$};
	\addplot[only marks, mark=*, color=red]
    		coordinates{
      			(1,1)
      			(1.75,2)
      			(3,3)
     			(4,4)
      			(5,5)
      			(6,6)
      			(7,7)
      			(8,8)
      			(9,9)
   		};
    	\addlegendentry{$\scriptstyle  \delta(L)=1, \; \mathrm{Aut}(\mathcal{X})<\infty, \; \text{no EFS}$};
      	\addplot[only marks, mark=x, thick, color=red]
    		coordinates{
      			(2.25,2)
      			(6.25,4) 
		};
     	\addlegendentry{$\scriptstyle  \delta(L)=0, \; \mathrm{Aut}(\mathcal{X})<\infty, \; \text{no EFS}$};   
	\addplot[only marks, mark=*, color=blue]
    		coordinates{
			(11,9) 
			(12,8) 		(12,10)
			(13,7)  		(13,9)
			(13.75,6) 		(14,8)
			(15,5) 		(15,7)
			(16,4) 		(16,6)
			(17,3) 		(17,5)
			(17.75,2) 		(17.75,4)
			(19,3)
			(20,2)
  		};
    	\addlegendentry{$\scriptstyle  \delta(L)=1, \; \mathrm{Aut}(\mathcal{X})=\infty$};
	\addplot[only marks, mark=x, thick, color=blue]
    		coordinates{
      			(10.25,8) 
			(14.25,6) 
			(18.25,2) 		(18.25,4)  
		};
     	\addlegendentry{$\scriptstyle  \delta(L)=0, \; \mathrm{Aut}(\mathcal{X})=\infty$};   
	\addplot[only marks, mark=*, color=cyan]
    		coordinates{
			(9.75,10)  (11,11) 
		};
     	\addlegendentry{$\scriptstyle  \delta(L)=1, \; \mathrm{Aut}(\mathcal{X})=\infty, \; \text{no EFS}$};   
	\addplot[only marks, mark=x, thick, color=cyan]
    		coordinates{
      			(10.25,10)  
	 	};
     	\addlegendentry{$\scriptstyle  \delta(L)=0, \; \mathrm{Aut}(\mathcal{X})=\infty, \; \text{no EFS}$};   
  \end{axis}
\end{tikzpicture}
\end{center}
\captionsetup{justification=centering}
\caption{Hyperbolic, even, 2-elementary lattices $L$ occurring as N\'eron-Severi lattice of a K3 surface $\mathcal{X}$
}
\label{fig:K3_2elet}
\end{figure}
\par Let us consider the case of hyperbolic even lattices $L$ that admit a primitive embedding $ H \hookrightarrow L$. When occurring as N\'eron-Severi lattice for a 
K3 surface $\mathcal{X}$, this condition corresponds to the existence of an {\it elliptic fibration with section} (EFS) or Jacobian elliptic fibration. In this case, one has a lattice 
isomorphism $ L \simeq H \oplus K $, where $K$ is an even, negative-definite lattice. By $K^{\text{root}}$ the sub-lattice spanned by the roots of $K$, i.e., the lattice elements of self-intersection $-2$  in $K$. The factor group is denoted by $\mathpzc{W} = K / K^{\text{root}}$. The following result is due to Nikulin (see \cite{MR633160b}*{Thm.~0.2.2}):
\begin{theorem} Let $L$ be a hyperbolic lattice occurring as N\'eron-Severi lattice for a K3 surface with finite automorphism group. Assume that $L \simeq H \oplus K $. Then: \\
(i) $L$ is 2-elementary if and only if:
\begin{itemize} 
\item [(a)]   $K=  E_8(-1)^{\oplus 2} \oplus A_1(-1)$, or 
\item [(b)]  $K$ is 2-elementary, with $ K \nsimeq E_8(-2)$  and $\operatorname{rank}{K} + \ell_K \le 16 $
\end{itemize} 
 (ii) $L$ is not 2-elementary if and only if $K(-1)$ is one of the following: 
 \begin{gather*}
 A_2; \; 
 A_1 \oplus A_2, A_3; \;
 A_1^{\oplus 2} \oplus A_2, A_2^{\oplus 2},  A_1 \oplus A_3, A_4; \; 
 A_1 \oplus A_2^{\oplus 2}, A_1^{\oplus 2} \oplus A_3,  A_1 \oplus A_4, \\ A_2 \oplus A_3, A_5, D_5; 
 A_2^{\oplus 3}, A_3^{\oplus 2}, A_2 \oplus A_4, A_1 \oplus A_5, A_6, A_2 \oplus D_4, A_1 \oplus D_5, E_6; \\
 A_7, A_3 \oplus D_4, A_2 \oplus D_5, D_7, A_1\oplus E_6; \;
 A_2 \oplus E_6; \;
 A_2 \oplus E_8; \;
 A_3 \oplus E_8.
\end{gather*}
Note that in (ii) we always have $\operatorname{rank}{K} \le 11$.
\end{theorem} 
\par Finally, for a \emph{Jacobian elliptic surface} $\mathcal{X}$ we shall denote the elliptic fibration by $\pi: \mathcal{X} \to \mathbb{P}^1$, the zero-section by $\sigma$, and the Mordell-Weil group of sections by $\operatorname{MW}(\mathcal{X}, \pi) $.  A complete list of the possible singular fibers in an associated Weierstrass model was given by Kodaira~\cite{MR0184257}: it encompasses two infinite families $(I_n, I_n^*, n \ge0)$ and six exceptional cases $(II, III, IV, II^*, III^*, IV^*)$. Upon resolving the singular fibers, we obtain the corresponding reducible fibers of a smooth model. For fibers of type $I_n$, $I_m^*$, $IV^*$, $II^*$, and $II^*$, one obtains reducible fibers of type $A_{n-1}$, $D_{m+4}$, $E_6$, $E_7$, and $E_8$, respectively. Similarly, singular fibers of type $III$ and $IV$ produces reducible fibers of type $A_1$ and $A_2$, respectively. Shioda \cite{MR1081832} proved that there is a canonical group isomorphism $\mathpzc{W} \simeq \operatorname{MW}(\mathcal{X}, \pi) $, identifying $\mathpzc{W}$ with the Mordell-Weil group of the corresponding Jacobian elliptic fibration $(\pi,\sigma)$. 
\section {The Finite Automorphism Group Case}
We have the following:
\begin{theorem}
\label{prop1}
Let $\pi: \mathcal{X} \to \mathbb{P}^1$ be a Jacobian elliptic K3 surface with N\'eron-Severi lattice $L$. The K3 surface $\mathcal{X}$ has a finite automorphism group if and only if $L$ appears in Tables~\ref{tab:K3JEF_a},~\ref{tab:K3JEF_b}. Then, $\mathcal{X}$ admits exactly the Jacobian elliptic fibration(s) with the reducible fibers and Mordell-Weil group $(K^{\text{root}}, \mathpzc{W})$ listed in the right column. 
\end{theorem}
\begin{remark}
Theorem~\ref{prop1} incorporates results from previous works, as follows:
\begin{enumerate}
\item For $(\rho, \ell, \delta) = (14, 2, 0), (14, 4, 0), (14, 4, 1)$ all Jacobian elliptic fibrations are constructed in \cites{Clingher:2020baq}. This work provides birational models for the corresponding K3 surfaces as quartic projective hypersurfaces. 
\item For $(\rho, \ell, \delta) = (16, 2, 1)$, all Jacobian elliptic fibrations are constructed in \cites{MR4160930, MR4544843}.
\item For $(\rho, \ell, \delta) = (17, 1, 1)$ the associated lattice polarized K3 surfaces are the Shioda-Inose surfaces (see \cites{MR728142}) covering the Kummer surfaces associated with the Jacobian of a general curve of genus two. The two possible Jacobian elliptic fibrations in this case are constructed and discussed in \cites{MR2427457, MR2824841, MR2935386, MR3366121,MR3712162}.
\item For $(\rho, \ell, \delta) = (18, 0, 0)$ the associated lattice polarized K3 surfaces are the Shioda-Inose K3 surfaces covering the Kummer surfaces associated with the product of two non-isogenous elliptic curves. Quartic surface models in this case are given in \cite{MR2369941}. The two Jacobian elliptic fibrations in this case are constructed in \cites{MR578868}.
\item Other examples of equations relating the elliptic fibrations of K3 surfaces with 2-elementary N\'eron-Severi lattices and double sextics or quartic hypersurfaces were provided in \cites{MR4130832, MR3882710}
\end{enumerate}
\end{remark}
\begin{proof}
If $\rho=1$ then the corresponding K3 surfaces have no elliptic fibrations. If $\rho=2$ then the automorphism groups of all Jacobian elliptic fibrations are evidently finite.  For 2-elementary lattices the results in \cites{MR1029967, MR633160b} prove that a general $L$-polarized K3 surface with $L$ appearing as entry in the left column have a finite automorphism group. Thus, any Jacobian elliptic fibration supported on a a K3 surface $\mathcal{X}$ with N\'eron-Severi lattice isomorphic to $L$ must have a Mordell-Weil group of vanishing rank. In \cite{MR1813537} Shimada provides a complete list of trivial lattices that exit for Jacobian elliptic K3 surfaces with Mordell-Weil groups of vanishing rank. Computing $\ell(L)$ and $\delta(L)$ for each 2-elementary lattice generated by these Jacobian elliptic fibrations we generate the matching entries in the right columns for all 2-elementary lattices. This result is then extended beyond the case of 2-elementary lattices using the results in \cite{MR633160b}*{Thm.~0.2.2} and \cite{MR3165023}. Only sporadic cases are added, for example the lattice $H\oplus E_6(-1) \oplus A_2(-1)$; no cases are added in rank bigger than or equal to 14. A detailed proof covering the cases of rank 14 (where several elliptic fibrations are supported for each length) was already given in \cites{MR4130832, Clingher:2020baq}.
\end{proof}
Let us also note the following:
\begin{corollary}
All K3 surfaces $\mathcal{X}$ with 2-elementary N\'eron-Severi lattice $L$ that do not support a Jacobian elliptic fibration are listed in Table~\ref{tab:K3JEFb}.
\end{corollary}
\begin{proof}
For $\rho \le 9$, Nikulin \cite{MR633160b} proved that the corresponding K3 surface has a finite automorphism group. However, we note that there is no Jacobian elliptic fibration in Shimada's list \cite{MR1813537} which realizes any of the lattices in question. For $\rho=10, 11$, one may use \cite{MR633160b}*{Thm.~0.2.2} to show that the lattices cannot be written as $H \oplus K$.
\end{proof}
\begin{table}[!ht]
\begin{tabular}{c|clc|c}
$\rho$ &  \multicolumn{1}{c}{$L$}  & $\det{q_L}$ & $\delta$ & $\mathrm{Aut}(\mathcal{X})$ \\
\hline
\hline
$1$ & $\langle 2 \rangle$ & $2^1$ & $1$ & finite  \\
\hline
$2$ & $H(2)$ & $2^2$ & $0$ & finite  \\
\hdashline
$2$ & $\langle 2 \rangle \oplus \langle -2 \rangle$ & $2^2$ & $1$ & finite  \\
\hline
$3$ & $H(2) \oplus A_1(-1)$ & $2^3$ & $1$ & finite \\
\hline
$4$ & $H(2) \oplus A_1(-1)^{\oplus 2}$ & $2^4$ & $1$ & finite  \\
\hline
$5$ & $H(2) \oplus A_1(-1)^{\oplus 3}$ & $2^5$ & $1$ & finite  \\
\hline
$6$ & $H(2) \oplus D_4(-1)$ & $2^4$ & $0$ & finite \\
\hdashline
$6$ & $H(2) \oplus A_1(-1)^{\oplus 4}$ & $2^6$ & $1$ & finite  \\
\hline
$7$ & $H(2) \oplus A_1(-1)^{\oplus 5}$ & $2^7$ & $1$ & finite \\
\hline
$8$ & $H(2) \oplus A_1(-1)^{\oplus 6}$ & $2^8$ & $1$ & finite  \\
\hline
$9$ & $H(2) \oplus A_1(-1)^{\oplus 7}$ & $2^9$ & $1$ & finite  \\
\hline
$10$ & $H(2) \oplus E_8(-2)$ & $2^{10}$ & $0$ & infinite \\
\hdashline
$10$ & $H(2) \oplus A_1(-1)^{\oplus 8}$ & $2^{10}$ & $1$ & infinite \\
\hline
$11$ & $H(2) \oplus A_1(-1)^{\oplus 9}$ & $2^{11}$ & $1$ & infinite \\
\hline
\end{tabular}
\captionsetup{justification=centering}
\caption{2-elementary K3 lattices not supporting a Jacobian elliptic fibration}
\label{tab:K3JEFb}
\end{table}
\section{The Infinite Automorphism Group Case}
We have the following:
\begin{theorem}
\label{prop2}
Let $\pi: \mathcal{X} \to \mathbb{P}^1$ be a Jacobian elliptic K3 surface with a N\'eron-Severi lattice $L$ that is 2-elementary. The K3 surface $\mathcal{X}$ has an infinite automorphism group if and only if $L$ appears in Tables~\ref{tab:K3JEF2_a}--\ref{tab:K3JEF2_c}. In each case, the Jacobian elliptic fibrations with finite Mordell-Weil group supported on $\mathcal{X}$ are indicated in the right column,  in terms of their root lattice $K^{\text{root}}$ and Mordell-Weil group $\mathpzc{W}$. Moreover, in each case $\mathcal{X}$ supports at least one additional Jacobian elliptic fibration with  $\operatorname{rank}\operatorname{MW}(\mathcal{X}, \pi) >0$.
\end{theorem}
\begin{remark}
Theorem~\ref{prop2} incorporates results from previous works, as follows:
\begin{enumerate}
\item For $(\rho, \ell, \delta) = (10, 8,0)$, the corresponding lattice-polarized K3 surfaces are double covers of a general rational elliptic surface and do not admit Jacobian elliptic fibrations with finite Mordell-Weil group. This K3 family is discussed in \cite{MR4069236}*{Sec.~4.5}.
\item For $(\rho, \ell, \delta) = (16, 6,1)$, the corresponding  K3 surface is a special type of double-sextic K3 surface. In this case, the possible Jacobian elliptic fibrations are classified by Kloosterman \cite{MR2254405}.
\item For $(\rho, \ell, \delta) = (16+k, 6-k, 1)$ with $k=0, 1, 2, 3$, the associated lattice polarized K3 surfaces are constructed, as twisted Legendre pencils, in \cites{MR4494119, MR1013162, MR894512, MR1023921, MR2854198, MR1877757}.
\item For $(\rho, \ell, \delta) = (18, 4,0)$,  the associated lattice-polarized K3 surfaces ares general Kummer surfaces associated with products of two non-isogenous elliptic curves. 
The Jacobian elliptic fibrations in this case are classified in \cites{MR1013073, MR2409557}.
\end{enumerate}
\end{remark}
\begin{proof}
We follow the same strategy as in the proof of Theorem~\ref{prop1}. For $(\rho, \ell) = (14, 8)$, the fibration with twelve singular fibers of type $I_2$ and Mordell-Weil group $(\mathbb{Z}/2\mathbb{Z})^2$ was discussed in detail in \cite{MR1871336}.  For $(\rho, \ell) = (14, 6)$ and $(\rho, \ell)=(18,2)$, one determines by explicit lattice theoretic calculations whether $\delta=0$ or $\delta=1$, in the case of fibrations with non-trivial Mordell-Weil group. This approach was discussed in \cites{Clingher:2020baq}. For $(\rho, \ell) = (18, 4)$ we apply results from \cites{MR1013073, MR2409557} to determine, in the case of fibrations with non-trivial Mordell-Weil group, whether $\delta=0$ or $\delta=1$.
\par  Finally, for each K3 surface $\mathcal{X}$ considered in this theorem there is at least one Jacobian elliptic fibration with $\operatorname{rank}\operatorname{MW}(\mathcal{X}, \pi)  >0$.  This follows from a theorem of Nikulin \cite{MR633160b}*{Thm.~10.2.1}. 
\end{proof}
\begin{remark}
In \cite{MR1432379} Nishiyama proposed an alternate way of classifying Jacobian elliptic fibrations, including the ones with infinite Mordell-Weil groups.  However, we will be completing the task of computing the lattice theoretic multiplicities of all Jacobian elliptic fibrations with finite Mordell-Weil group in this article. The classification of the Jacobian elliptic fibrations with an infinite Mordell-Weil group for the K3 surfaces in Tables~\ref{tab:K3JEF2_a}--\ref{tab:K3JEF2_c} will be presented elsewhere.
\end{remark}
\section {Multiplicities of the Jacobian elliptic fibrations}
\label{multiplicities}
Lastly, let us investigate the number of distinct primitive lattice embeddings $H \hookrightarrow L$, following the approach in \cite{FestiVeniani20}. Assume $j \colon H \hookrightarrow L$ is such a primitive embedding. Denote by $K$ the orthogonal complement of $j(H)$ in $L$ and denote by $K^{\text{root}}$ the sub-lattice spanned by the roots of $K$, i.e., the lattice elements of self-intersection $-2$  in $K$. The factor group is denoted by $\mathpzc{W} = K / K^{\text{root}}$. It follows that $L = j(H) \oplus K$. Moreover, the lattice $K$ is negative-definite of rank $\ell = \rho -2$, and its discriminant group and pairing satisfy
\beq
\label{discr1}
 \Big( D( K ) , \, q_K  \Big) \ \simeq \ \Big( D(L), \, q_L \Big)  \,,
\eeq
with $D(L) \cong  \mathbb{Z}_2 ^\ell$. Now assume that for a K3 surface $\mathcal{X}$ with $\operatorname{NS}(\mathcal{X})=L$ we have a second primitive embedding $j' \colon H \hookrightarrow L$, such that the orthogonal complement of the image $j'(H)$, denoted $K'$, is isomorphic to $K$. We would like to see under what conditions $j$ and $j'$ correspond to Jacobian elliptic fibrations isomorphic under $\mathrm{Aut}(\mathcal{X}) $. By standard lattice-theoretic arguments (see  \cite{MR525944}*{Prop.~ 1.15.1}), there will exist an isometry $ \gamma \in O(L) $ such that $j' = \gamma \circ j$. The isometry $\gamma$ has a counterpart $ \gamma^* \in O(D(K)) $ obtained as image of $\gamma$ under the group homomorphism
\beq 
\label{izo22}
O(L) \ \rightarrow \  
O\big(D(L)\big) \ \simeq \ O\big(D(K)\big) \,.
\eeq
The isomorphism in (\ref{izo22}) is due to the decomposition $L = j(H) \oplus K$ and, as such, it depends on the lattice embedding $j$. 
\par Denote the group $O(D(K)) $ by $\mathpzc{A}$. There are two subgroups of $\mathpzc{A}$ that are relevant to our discussion. The first subgroup $\mathpzc{B} \leqslant \mathpzc{A}$  is given as the image of the following group homomorphism:
\beq
\label{eqn:map_B}
O(K) \ \simeq \  
\big \{ 
\varphi \in O(L) \ \vert \  \varphi \circ j(H) = j(H) 
\big \} 
 \ \hookrightarrow \ O(L) \ \rightarrow \ O\big(D(L)\big) \ \simeq \ O\big(D(K)\big)  \ .
\eeq
The second subgroup $\mathpzc{C} \leqslant \mathpzc{A}$ is obtained as the image of following group homomorphism:
\beq
O_h(\mathrm{T}_{\mathcal{X}}) \ \hookrightarrow \  O(\mathrm{T}_{\mathcal{X}}) 
\ \rightarrow \ O \big( D ( \mathrm{T}_{\mathcal{X}} ) \big)  \ \simeq \ 
 O\big(D(L)\big) \ \simeq \ O\big(D(K)\big)  \ .
\eeq
Here $\mathrm{T}_{\mathcal{X}} $ denotes the transcendental lattice of the K3 surface $\mathcal{X}$ and $O_h(\mathrm{T}_{\mathcal{X}})$ is given by the isometries of 
$\mathrm{T}_{\mathcal{X}} $ that preserve the Hodge decomposition. Furthermore, one has $ D(\operatorname{NS}(\mathcal{X})) \simeq D(\mathrm{T}_{\mathcal{X}})  $  with $q_L = -q_{\mathrm{T}_{\mathcal{X}}} $, as $\operatorname{NS}(\mathcal{X})=L$ and $\mathrm{T}_{\mathcal{X}}$ is the orthogonal complement of $\operatorname{NS}(\mathcal{X}) $ with respect to a unimodular lattice. 
\par Consider then the correspondence
\beq
\label{corresp33} 
 H \  \overset{j}{\hookrightarrow} \ L \qquad  \rightsquigarrow \qquad \mathpzc{C} \,  \gamma^* \mathpzc{B} \,,
 \eeq
 that associates to a lattice embedding $H \hookrightarrow L$ a double coset in $\mathpzc{C}  \backslash \mathpzc{A}/ \mathpzc{B}$. As proved in  \cite{FestiVeniani20}*{Thm~2.8},  the map (\ref{corresp33}) establishes a one-to-one correspondence between Jacobian elliptic fibrations on $\mathcal{X}$ with $ j(H)^{\perp} \simeq K$, up to the action of the automorphism group ${\rm Aut}(\mathcal{X}) $ and the elements of the double coset set $\mathpzc{C}  \backslash \mathpzc{A}/ \mathpzc{B}$. The number of elements in the double coset is referred by Festi and Veniani as the \emph{multiplicity} associated with the frame $(K^{\text{root}}, \mathpzc{W})$.
 \begin{remark}
 From \cite{MR1892313} it follows that $O_h(\mathrm{T}_{\mathcal{X}})  = \{ \pm \mathrm{id} \}$ for a K3 surface $\mathcal{X}$ of odd Picard number. The same is true for a K3 surface $\mathcal{X}$ of even Picard number if the period $\omega_\mathcal{X} \in \mathrm{T}_{\mathcal{X}} \otimes \mathbb{C}$ is very general. One then has $| \mathcal{C} | =1$, and the multiplicity of a frame $(K^{\text{root}}, \mathpzc{W})$ equals  $|\mathpzc{A}|/|\mathpzc{B}|$.
  \end{remark}
\begin{remark} 
$\mathpzc{A}$ is the orthogonal group of  the torsion quadratic form $q_K$. This is different from the orthogonal group of the associated bilinear form.  In the case $\rho = \ell +2$ and $L= H \oplus K$ with $K(-1) = A_1^{\oplus \ell}$, the orthogonal group $O(\ell, \mathbb{F}_2)$ consists of elements $g \in \mathrm{GL}(\ell, \mathbb{F}_2)$ with $g\cdot g^t = \mathbb{I}_\ell$.  The order $|O(\ell, \mathbb{F}_2)|$ was computed in \cite{MR1189139} and differs from $|\mathpzc{A}|$ for $\ell \ge 4$.
\end{remark}
 We have the following:
 \begin{theorem}
 \label{prop3}
In Tables~ \ref{tab:K3JEF_a}-\ref{tab:K3JEF2_c} for each frame $(K^{\text{root}}, \mathpzc{W})$ the ratio $|\mathpzc{A}|/|\mathpzc{B}|$ is given. In particular, the lattice embedding $H \hookrightarrow L$ is unique for each frame with $|\mathpzc{A}|/|\mathpzc{B}|=1$.
\end{theorem}
\begin{remark}
In~Tables~\ref{tab:max_vects_Z2_finite}-\ref{tab:max_vects_Z2xZ2_infinite} we provide vectors that together with $K^{\text{root}}$ span the overlattice $K$ for $\mathpzc{W}=\mathbb{Z}/2\mathbb{Z}$  and $\mathpzc{W}=(\mathbb{Z}/2\mathbb{Z})^2$, respectively. If for a given root lattice there exist distinct overlattices $K, K'$ the entries are printed in {\color{blue} blue} if $H \oplus K \cong H \oplus K'$ and in {\color{green} green} if $H \oplus K \not \cong H \oplus K'$. Distinct overlattices arise when the section and 2-torsion section of a Jacobian elliptic fibration can intersect end-nodes of a reducible fiber of type $\widetilde{D}_n$ with $n>4$ either on the same fork or on opposite forks.
\end{remark}
\begin{remark}
For the cases $(\rho, \ell, \delta) = (16,6, 1), (18,2,0), (18,4,0)$ in Table~\ref{tab:K3JEF2_b} the ratio of orders $|\mathpzc{A}|/|\mathpzc{B}|$ is in agreement with the multiplicities as computed in \cite{FestiVeniani20}.
\end{remark}
\begin{proof}
The statement is proven by a computation in the {\sc Sage} class {\tt QuadraticForm}.   For a given lattice {\tt K}, the discriminant group is computed using the command {\tt D = K.discriminant\char`_group()}. The automorphism groups are computed using
\beqn
\text{{\tt O=K.orthogonal\char`_group()} and {\tt A=D.orthogonal\char`_group()}.}
\eeqn
Images of the generators are computed using {\tt B=D.orthogonal\char`_group(O.gens())}.
\par For $\mathpzc{W}=\mathbb{Z}/2\mathbb{Z}$  the lattice $K$ is the overlattice spanned by $K^{\text{root}}$ and one additional lattice vector $\vec{v}_{\text{max}}$.  For $\mathpzc{W}=(\mathbb{Z}/2\mathbb{Z})^2$ one has two additional vectors $\vec{v}_{\text{max}, i}$ with $i=1, 2$. The lattice $K$ is the overlattice spanned by $K^{\text{root}}$ and the additional lattice vectors $\vec{v}_{\text{max}, 1}, \vec{v}_{\text{max}, 2}$. These vectors are in the orthogonal complement of the sub-lattice spanned by the section $\sigma$ and the smooth fiber class $F$. The lattice vectors can be computed using the properties of the elliptic fibration. Using the same notation for the bases of root lattices as in \cite{MR1813537}*{Sec.~6} and ordering the bases by ADE-type and from lowest to highest rank, the vectors $\vec{v}_{\text{max}}$ are given in~Table~\ref{tab:max_vects_Z2_finite} for the overlattices (with finite automorphism group)  in Tables~\ref{tab:K3JEF_a},~\ref{tab:K3JEF_b}.  The vectors $\vec{v}_{\text{max}}$ are given in~Tables~\ref{tab:max_vects_Z2_infinite_a}, \ref{tab:max_vects_Z2_infinite_b} for the overlattices (with infinite automorphism group) in Tables~\ref{tab:K3JEF2_a}-\ref{tab:K3JEF2_c}.  The vectors $\vec{v}_{\text{max},1}, \vec{v}_{\text{max},2}$ are given in~Table~\ref{tab:max_vects_Z2xZ2_infinite} for the overlattices (with infinite automorphism group) in Tables~\ref{tab:K3JEF2_a}-\ref{tab:K3JEF2_c}. Each overlattice is then computed using the command {\tt overlattice}. The corresponding Gram matrix is computed using the command {\tt gram\char`_matrix}.  The orders of the automorphism groups are computed using the command {\tt order}. In the case $\rho = \ell +2$ and $L= H \oplus K$ with $K(-1) = A_1^{\oplus \ell}$ and $\ell=9, 10$ we used {\sc Magma} to compute $|\mathpzc{A}|$.
\par For each lattice with $|\mathpzc{A}|/|\mathpzc{B}|=1$ one then checks that  the images of the generators for $O(K)$ also generate $O(D(K))$. Thus, the map in Equation~(\ref{eqn:map_B}) is surjective and the multiplicity associated with $(K^{\text{root}}, \mathpzc{W})$ equals one.
\end{proof}
\begin{table}[!ht]
\scalemath{0.8}{
\begin{tabular}{c|clc|cc|c}
$\rho$ &  \multicolumn{1}{c}{$L= H \oplus K$}  & $\det{q_L}$ & $\delta$ & $K^{\text{root}}(-1)$ & $\mathpzc{W}$ & $|\mathpzc{A}|/|\mathpzc{B}|$\\
\hline
\hline
$2$	& $H$ 								& $2^0$ 		& $0$	& -- 				& $\{ \mathbb{I} \}$ 						& --\\
\hline
$3$	& $H \oplus A_1(-1)$						& $2^1$ 		& $1$	& $A_1$ 			& $\{ \mathbb{I} \}$ 						& 1\\
\hline
$4$ 	& $H \oplus A_1(-1)^{\oplus 2}$				& $2^2$ 		& $1$	& $2 A_1$			& $\{ \mathbb{I} \}$ 						& 1\\
\hdashline
$4$	& $H \oplus A_2$ 						& $3$ 		& -- 		& $A_2$ 			& $\{ \mathbb{I} \}$ 						& 1\\
\hline
$5$	& $H\oplus A_1(-1)^{\oplus 3}$				& $2^3$		& $1$	& $3 A_1$			& $\{ \mathbb{I} \}$ 						& 1\\
\hdashline
$5$	& $H\oplus A_3(-1)$ 						& $2^2$ 		& -- 		& $A_3$ 			& $\{ \mathbb{I} \}$ 						& 1\\
\hdashline
$5$ 	& $H\oplus A_2(-1) \oplus A_1(-1)$ 			& $2 \cdot 3$	& -- 		& $A_2 + A_1$		& $\{ \mathbb{I} \}$ 						& 1\\
\hline
$6$	& $H\oplus D_4(-1)$						& $2^2$ 		& $0$ 	& $D_4$ 			& $\{ \mathbb{I} \}$ 						& 1\\
\hdashline
$6$	& $H\oplus A_1(-1)^{\oplus 4}$ 				& $2^4$ 		& $1$ 	& $4 A_1$ 		& $\{ \mathbb{I} \}$ 						& 1\\
\hdashline
$6$	& $H\oplus A_4(-1)$						& $5$ 		& -- 		& $A_4$ 			& $\{ \mathbb{I} \}$ 						& 1\\
\hdashline
$6$	& $H\oplus A_3(-1) \oplus A_1(-1)$ 			& $2^3$ 		& -- 		& $A_3 + A_1$ 		& $\{ \mathbb{I} \}$ 						& 1\\
\hdashline
$6$ 	& $H\oplus A_2(-1)^{\oplus 2}$ 				& $3^2$ 		& -- 		& $2 A_2$ 		& $\{ \mathbb{I} \}$ 						& 1\\
\hdashline
$6$	& $H\oplus A_2(-1) \oplus A_1(-1)^{\oplus 2}$	& $2^2 \cdot 3$	& --		& $A_2 + 2 A_1$	& $\{ \mathbb{I} \}$ 						& 1\\
\hline
$7$ 	& $H\oplus D_4(-1) \oplus A_1(-1)$ 			& $2^3$ 		& $1$ 	& $D_4 + A_1$ 	& $\{ \mathbb{I} \}$ 						& 1\\
\hdashline
$7$	& $H\oplus A_1(-1)^{\oplus 5}$ 				& $2^5$ 		& $1$ 	& $5 A_1$ 		& $\{ \mathbb{I} \}$ 						& 1\\
\hdashline
$7$	& $H\oplus D_5(-1)$						& $2^2$ 		& -- 		& $D_5$ 			& $\{ \mathbb{I} \}$ 						& 1\\
\hdashline
$7$ 	& $H\oplus A_5(-1)$ 						& $2\cdot 3$ 	& -- 		& $A_5$ 			& $\{ \mathbb{I} \}$						& 1\\
\hdashline
$7$ 	& $H\oplus A_4(-1)  \oplus A_1(-1)$ 			& $2\cdot 5$ 	& -- 		& $A_4 + A_1$ 		& $\{ \mathbb{I} \}$ 						& 1\\
\hdashline
$7$	& $H\oplus A_3(-1)  \oplus A_2(-1)$ 			& $2^2\cdot 3$	& -- 		& $A_3 + A_2$ 		& $\{ \mathbb{I} \}$ 						& 1\\
\hdashline
$7$	& $H\oplus A_3(-1)  \oplus A_1(-1)^{\oplus 2}$	& $2^4$ 		& -- 		& $A_3 + 2 A_1$ 	& $\{ \mathbb{I} \}$ 						& 1\\
\hdashline
$7$ 	& $H\oplus A_2(-1)^{\oplus 2}  \oplus A_1(-1)$ 	& $2 \cdot 3^2$	& -- 		& $2 A_2 + A_1$ 	& $\{ \mathbb{I} \}$ 						& 1\\
\hline
$8$	& $H\oplus D_6(-1)$ 					& $2^2$		& $1$ 	& $D_6$ 			& $\{ \mathbb{I} \}$ 						& 1\\
\hdashline
$8$	& $H\oplus D_4(-1) \oplus A_1(-1)^{\oplus 2}$ 	& $2^4$ 		& $1$ 	& $D_4 + 2 A_1$ 	& $\{ \mathbb{I} \}$ 						& 1\\
\hdashline
$8$ 	& $H\oplus A_1(-1)^{\oplus 6}$ 				& $2^6$ 		& $1$ 	& $6 A_1$ 		& $\{ \mathbb{I} \}$ 						& 2\\
\hdashline
$8$	& $H\oplus E_6(-1)$ 						& $3$ 		& -- 		& $E_6$ 			& $\{ \mathbb{I} \}$ 						& 1\\
\hdashline
$8$	& $H\oplus D_5(-1) \oplus A_1(-1)$ 			& $2^3$ 		& -- 		& $D_5 + A_1$ 	& $\{ \mathbb{I} \}$ 						& 1\\
\hdashline
$8$	& $H\oplus D_4(-1) \oplus A_2(-1)$ 			& $2^2 \cdot 3$	& -- 		& $D_4 + A_2$ 	& $\{ \mathbb{I} \}$ 						& 1\\
\hdashline
$8$	& $H\oplus A_6(-1)$ 						& $7$ 		& -- 		& $A_6$ 			& $\{ \mathbb{I} \}$ 						& 1\\
\hdashline
$8$	& $H\oplus A_5(-1) \oplus A_1(-1)$ 			& $2^2 \cdot 3$	& -- 		& $A_5 + A_1$ 		& $\{ \mathbb{I} \}$ 						& 1\\
\hdashline
$8$	& $H\oplus A_4(-1) \oplus A_2(-1)$ 			& $3 \cdot 5$ 	& -- 		& $A_4 + A_2$ 		& $\{ \mathbb{I} \}$ 						& 1\\
\hdashline
$8$ 	& $H\oplus A_3(-1)^{\oplus 2}$ 				& $2^4$ 		& -- 		& $2 A_3$ 		& $\{ \mathbb{I} \}$ 						& 1\\
\hdashline
$8$ 	& $H\oplus A_2(-1)^{\oplus 3}$ 				& $3^3$ 		& -- 		& $3 A_2$ 		& $\{ \mathbb{I} \}$ 						& 1\\
\hline
$9$ 	& $H\oplus E_7(-1)$ 						& $2$ 		& $1$ 	& $E_7$ 			& $\{ \mathbb{I} \}$ 						& 1\\
\hdashline
$9$	& $H\oplus D_6(-1) \oplus A_1(-1)$ 			& $2^3$ 		& $1$ 	& $D_6 + A_1$ 	& $\{ \mathbb{I} \}$ 						& 1\\
\hdashline
$9$	& $H\oplus D_4(-1) \oplus A_1(-1)^{\oplus 3}$	& $2^5$ 		& $1$ 	& $D_4 + 3 A_1$ 	& $\{ \mathbb{I} \}$ 						& 2\\
\hdashline
$9$	& $H\oplus A_1(-1)^{\oplus 7}$ 				& $2^7$ 		& $1$ 	& $7 A_1$ 		& $\{ \mathbb{I} \}$ 						& 8\\
\hdashline
$9$	& $H\oplus E_6(-1) \oplus A_1(-1)$ 			& $2\cdot 3$ 	& -- 		& $E_6 + A_1$ 		& $\{ \mathbb{I} \}$ 						& 1\\
\hdashline
$9$	& $H\oplus D_7(-1)$ 					& $2^2$ 		& -- 		& $D_7$			& $\{ \mathbb{I} \}$ 						& 1\\
\hdashline
$9$	& $H\oplus D_5(-1) \oplus A_2(-1)$ 			& $2^2\cdot 3$	& -- 		& $D_5 + A_2$ 	& $\{ \mathbb{I} \}$ 						& 1\\
\hdashline
$9$	& $H\oplus D_4(-1) \oplus A_3(-1)$ 			& $2^4$ 		& -- 		& $D_4 + A_3$ 	& $\{ \mathbb{I} \}$ 						& 1\\
\hdashline
$9$	& $H\oplus A_7(-1)$ 						& $2^3$ 		& -- 		& $A_7$ 			& $\{ \mathbb{I} \}$ 						& 1\\
\hline
$10$	& $H\oplus E_8(-1)$ 						& $1$ 		& $0$ 	& $E_8$ 			& $\{ \mathbb{I} \}$ 						& 1\\
\hdashline	
$10$	& $H\oplus D_8(-1)$ 					& $2^2$ 		& $0$ 	& $D_8$ 			& $\{ \mathbb{I} \}$						& 1\\
\hdashline
$10$	& $H\oplus E_7(-1) \oplus A_1(-1)$ 			& $2^2$ 		& $1$ 	& $E_7 + A_1$ 		& $\{ \mathbb{I} \}$						& 1\\
\hdashline
$10$	& $H\oplus D_4(-1)^{\oplus 2}$ 			& $2^4$ 		& $0$ 	& $2 D_4$ 		& $\{ \mathbb{I} \}$ 						& 1\\
\hdashline
$10$	& $H\oplus D_6(-1) \oplus A_1(-1)^{\oplus 2}$ 	& $2^4$ 		& $1$ 	& $D_6 + 2 A_1$ 	& $\{ \mathbb{I} \}$ 						& 2\\
\hdashline
$10$	& $H(2)\oplus D_4(-1)^{\oplus 2}\cong H \oplus N$ & $2^6$ 	& $0$ 	& $8 A_1$ 		& $\mathbb{Z}/2\mathbb{Z}$  				& 1\\
\hdashline
$10$	& $H\oplus D_4(-1) \oplus A_1(-1)^{\oplus 4}$ 	& $2^6$ 		& $1$ 	& $D_4 + 4 A_1$ 	& $\{ \mathbb{I} \}$						& 8\\
\hdashline
$10$	& $H\oplus A_1(-1)^{\oplus 8}$ 				& $2^8$ 		& $1$ 	& $8 A_1$ 		& $\{ \mathbb{I} \}$ 						& 64\\
\hdashline
$10$	& $H\oplus E_6(-1) \oplus A_2(-1)$ 			& $3^2$ 		& -- 		& $E_6 + A_2$ 		& $\{ \mathbb{I} \}$ 						& 1\\
\hline
\end{tabular}}
\captionsetup{justification=centering}
\caption{Jacobian elliptic K3 surfaces with $\mathrm{Aut}(\mathcal{X})< \infty$}
\label{tab:K3JEF_a}
\end{table}
\begin{table}[!ht]
\scalemath{0.8}{
\begin{tabular}{c|clc|cc|c}
$\rho$ &  \multicolumn{1}{c}{$L= H \oplus K$}  & $\det{q_L}$ & $\delta$ & $K^{\text{root}}(-1)$ & $\mathpzc{W}$ & $|\mathpzc{A}|/|\mathpzc{B}|$\\
\hline
\hline
$11$ 	& $H\oplus E_8(-1) \oplus A_1(-1)$ 				& $2^1$ 	& $1$ 	& $E_8 + A_1$ 		& $\{ \mathbb{I} \}$					& 1\\
\hdashline
$11$ 	& $H\oplus E_7(-1) \oplus A_1(-1)^{\oplus 2}$ 		& $2^3$ 	& $1$ 	& $E_7 + 2 A_1$ 	& $\{ \mathbb{I} \}$					& 1\\
& $\quad \cong H \oplus D_8(-1) \oplus A_1(-1)$ 			& 		& 		& $D_8 + A_1$ 	& $\{ \mathbb{I} \}$					& 1\\
\hdashline
$11$ 	& $H\oplus D_6(-1) \oplus A_1(-1)^{\oplus 3}$ 		& $2^5$ 	& $1$ 	& $D_6 + 3 A_1$ 	& $\{ \mathbb{I} \}$					& 6\\ 
& $\quad \cong H \oplus D_4(-1)^{\oplus 2} \oplus A_1(-1)$ 	& 		& 		& $2 D_4 + A_1$ 	& $\{ \mathbb{I} \}$					& 1\\
\hdashline
$11$ 	& $H\oplus D_4(-1) \oplus A_1(-1)^{\oplus 5}$ 		& $2^7$ 	& $1$ 	& $D_4 + 5 A_1$ 	& $\{ \mathbb{I} \}$					& 56\\
& 												& 		& 		& $9 A_1$ 		& $\mathbb{Z}/2\mathbb{Z}$			& 1\\
\hline
$12$ 	& $H\oplus E_8(-1) \oplus A_1(-1)^{\oplus 2}$ 		& $2^2$ 	& $1$ 	& $E_8 + 2 A_1$ 	& $\{ \mathbb{I} \}$					& 1\\
& $\quad \cong H \oplus D_{10}(-1)$ 					& 		& 		& $D_{10}$ 		& $\{ \mathbb{I} \}$					& 1\\
\hdashline
$12$ 	& $H\oplus E_7(-1) \oplus A_1(-1)^{\oplus 3}$ 		& $2^4$ 	& $1$ 	& $E_7 + 3 A_1$ 	& $\{ \mathbb{I} \}$					& 2\\
& $\quad \cong H \oplus D_8(-1)  \oplus A_1(-1)^{\oplus 2}$ 	& 		& 		& $D_8 + 2  A_1$ 	& $\{ \mathbb{I} \}$					& 3\\
& $\quad \cong H \oplus D_6(-1)  \oplus D_4(-1)$ 			& 		& 		& $D_6 + D_4$ 	& $\{ \mathbb{I} \}$					& 1\\
\hdashline
$12$ 	& $H\oplus D_6(-1) \oplus A_1(-1)^{\oplus 4}$ 		& $2^6$ 	& $1$ 	& $D_6 + 4 A_1$ 	& $\{ \mathbb{I} \}$					& 30\\ 
& $\quad \cong H \oplus D_4(-1)^{\oplus 2}  \oplus A_1(-1)^{\oplus 2}$ & 	& 		& $2 D_4 + 2 A_1$ 	& $\{ \mathbb{I} \}$					& 10\\ 
&												&		&		& $D_4 + 6 A_1$ 	& $\mathbb{Z}/2\mathbb{Z}$ 			& 1\\
\hdashline
$12$ 	& $H\oplus E_8(-1) \oplus A_2(-1)$ 				& $3$ 	& -- 		& $E_8 + A_2$ 		& $\{ \mathbb{I} \}$					& 1\\
\hline
$13$ 	& $H\oplus E_8(-1) \oplus A_1(-1)^{\oplus 3}$ 		& $2^3$ 	& $1$ 	& $E_8 + 3 A_1$ 	& $\{ \mathbb{I} \}$					& 1\\
& $\quad \cong H \oplus E_7(-1)\oplus D_4(-1)$ 			& 		& 		& $E_7 + D_4$ 	& $\{ \mathbb{I} \}$					& 1\\
& $\quad \cong H \oplus D_{10}(-1)\oplus A_1(-1)$ 			& 		& 		& $D_{10} + A_1$ 	& $\{ \mathbb{I} \}$					& 3\\ 
\hdashline
$13$ 	& $H\oplus E_7(-1) \oplus A_1(-1)^{\oplus 4}$ 		& $2^5$ 	& $1$ 	& $E_7 + 4 A_1$ 	& $\{ \mathbb{I} \}$					& 5\\ 
& $\quad \cong H \oplus D_8(-1)  \oplus A_1(-1)^{\oplus 3}$ 	& 		& 		& $D_8 + 3 A_1$ 	& $\{ \mathbb{I} \}$					& 10\\ 
& $\quad \cong H \oplus D_6(-1)  \oplus D_4(-1) \oplus A_1(-1)$ & 		& 		& $D_6 + D_4 + A_1$ & $\{ \mathbb{I} \}$					& 10\\ 
&												&		&		& $D_6 + 5 A_1$ 	& $\mathbb{Z}/2\mathbb{Z}$ 			& 1\\
\hdashline
$13$ 	& $H\oplus E_8(-1) \oplus A_3(-1)$ 				& $2^2$ 	& -- 		& $E_8 + A_3$ 		& $\{ \mathbb{I} \}$ 					& 1\\
&												&		&		& $D_{11}$ 		& $\{ \mathbb{I} \}$ 					& 1\\
\hline
$14$ 	& $H\oplus E_8(-1) \oplus D_4(-1)$ 				& $2^2$ 	& $0$ 	& $E_8 + D_4$ 	& $\{ \mathbb{I} \}$					& 1\\
& $\quad \cong H \oplus D_{12}(-1)$ 					& 		& 		& $D_{12}$ 		& $\{ \mathbb{I} \}$					& 3\\ 
\hdashline
$14$ 	& $H\oplus D_8(-1) \oplus D_4(-1)$ 				& $2^4$ 	& $0$ 	& $D_8 + D_4$ 	& $\{ \mathbb{I} \}$ 					& 10\\
&												&		& 		& $E_7 + 5 A_1$ 	& $\mathbb{Z}/2\mathbb{Z}$ 			& 1\\
\hdashline
$14$ 	& $H\oplus E_8(-1) \oplus A_1(-1)^{\oplus 4}$ 		& $2^4$ 	& $1$ 	& $E_8 + 4 A_1$ 	& $\{ \mathbb{I} \}$					& 1\\
& $\quad \cong H \oplus E_7(-1) \oplus D_4(-1) \oplus A_1(-1)$& 		& 		& $E_7 + D_4 + A_1$& $\{ \mathbb{I} \}$					& 4\\
& $\quad \cong H \oplus D_{10}(-1) \oplus A_1(-1)^{\oplus 2}$ 	& 		& 		& $D_{10} + 2 A_1$ 	& $\{ \mathbb{I} \}$					& 6\\ 
& $\quad \cong H \oplus D_{6}(-1)^{\oplus 2}$ 				& 		& 		& $2 D_6$ 		& $\{ \mathbb{I} \}$					& 3\\ 
&												&		&		& $D_8 + 4 A_1$ 	& $\mathbb{Z}/2\mathbb{Z}$ 			& 1\\
\hline
$15$		& $H \oplus E_8(-1)  \oplus D_4(-1) \oplus A_1(-1)$ 	& $2^3$ 	& $1$ 	& $E_8 + D_4 + A_1$& $\{ \mathbb{I} \}$					& 1\\
& $\cong H \oplus E_7(-1) \oplus D_6(-1)$ 				&		&		& $E_7 + D_6$ 	& $\{ \mathbb{I} \}$					& 3\\
& $\cong H \oplus D_{12}(-1) \oplus A_1(-1)$ 				&		&		& $D_{12} + A_1$ 	& $\{ \mathbb{I} \}$					& 3\\
&												&		&		& $D_{10} + 3 A_1$	&$\mathbb{Z}/2\mathbb{Z}$ 			& 1\\
\hline
$16$ 	& $H \oplus E_8(-1)  \oplus D_6(-1)$ 			& $2^2$ 	& $1$ 	& $E_8 + D_6$ 	& $\{ \mathbb{I} \}$ 					& 1\\
& $\cong H \oplus E_7(-1) \oplus E_7(-1)$ 				&		&		& $2 E_7$ 		& $\{ \mathbb{I} \}$ 					& 1\\
& $\cong H \oplus D_{14}(-1)$ 							&		&		& $D_{14}$ 		& $\{ \mathbb{I} \}$ 					& 1\\
&												&		&		& $D_{12} + 2 A_1$ 	&$\mathbb{Z}/2\mathbb{Z}$ 			& 1\\
\hline
$17$ 	& $H \oplus E_8(-1)  \oplus E_7(-1)$ 			& $2^1$ 	& $1$ 	& $E_8 + E_7$ 	& $\{ \mathbb{I} \}$ 					& 1\\
& 												& 		& 		& $D_{14} + A_1$ 	& $\mathbb{Z}/2\mathbb{Z}$ 			& 1\\
\hline
$18$ 	& $H \oplus E_8(-1)  \oplus E_8(-1)$ 			& $2^0$ 	& $0$ 	& $2 E_8$ 		& $\{ \mathbb{I} \}$ 					& 1\\
& 												& 		& 		& $D_{16}$ 		& $\mathbb{Z}/2\mathbb{Z}$ 			& 1\\
\hline
$19$ & $H \oplus E_8(-1)  \oplus E_8(-1)  \oplus A_1(-1)$ 	& $2^1$ 	& $1$ 	& $2 E_8 + A_1$ 	& $\{ \mathbb{I} \}$					& 1\\
& 												& 		& 		& $E_7 + D_{10}$ 	& $\mathbb{Z}/2\mathbb{Z}$ 			& 1\\
& 												& 		& 		& $D_{16} + A_1$ 	& $\mathbb{Z}/2\mathbb{Z}$ 			& 1\\
\hline
\end{tabular}}
\captionsetup{justification=centering}
\caption{Jacobian elliptic K3 surfaces with $\mathrm{Aut}(\mathcal{X})< \infty$ (cont'd)}
\label{tab:K3JEF_b}
\end{table}
\begin{table}[!ht]
\scalemath{0.8}{
\begin{tabular}{c|clc|cc|c}
$\rho$ &  \multicolumn{1}{c}{$L= H \oplus K$}  & $\det{q_L}$ & $\delta$ & $K^{\text{root}}(-1)$ & $\mathpzc{W}$ & $|\mathpzc{A}|/|\mathpzc{B}|$\\
\hline
\hline
$10$		& $H \oplus E_8(-2)$ 								& $2^8$ 	& $0$	& --						& $E_8(2)$				& --\\
\hline
$11$ 	& $H\oplus A_1(-1)^{\oplus 9}$ 							& $2^9$ 	& $1$ 	& $9 A_1$ 				& $\{ \mathbb{I} \}$			& 960\\
\hline
$12$ 	& $H\oplus D_4(-1) \oplus A_1(-1)^{\oplus 6}$ 				& $2^8$ 	& $1$ 	& $D_4 + 6 A_1$ 			& $\{ \mathbb{I} \}$			& 672\\
& 														& 		& 		& $10 A_1$ 				& $\mathbb{Z}/2\mathbb{Z}$	& 36\\
\hdashline
$12$ 	& $H\oplus A_1(-1)^{\oplus 10}$ 						& $2^{10}$& $1$ 	& $10 A_1$ 				& $\{ \mathbb{I} \}$			& 26112\\
\hline
$13$ 	& $H\oplus D_6(-1) \oplus A_1(-1)^{\oplus 5}$ 				& $2^7$ 	& $1$ 	& $D_6 + 5 A_1$ 			& $\{ \mathbb{I} \}$			& 216\\
& $\quad \cong H\oplus D_4(-1)^{\oplus 2} \oplus A_1(-1)^{\oplus 3}$ 	& 		&		& $2 D_4 + 3 A_1$ 			& $\{ \mathbb{I} \}$			& 120\\
& 														& 		& 		& $D_4+ 7 A_1$ 			& $\mathbb{Z}/2\mathbb{Z}$	& 36\\
\hdashline
$13$ 	& $H\oplus D_4(-1) \oplus A_1(-1)^{\oplus 7}$ 				& $2^9$ 	& $1$ 	& $D_4 + 7 A_1$ 			& $\{ \mathbb{I} \}$			& 13056\\
& 														& 		& 		& $11 A_1$ 				& $\mathbb{Z}/2\mathbb{Z}$	& 1632\\
\hline
$14$ 	& $H\oplus D_4(-1)^{\oplus 3}$ 						& $2^6$ 	& $0$ 	& $3 D_4$ 				& $\{ \mathbb{I} \}$			& 40\\
&	 													&		& 		& {\color{green}$D_6 + 6 A_1$}	& $\mathbb{Z}/2\mathbb{Z}$	& 36\\
\hdashline
$14$		& $H\oplus D_6(-1) \oplus D_4(-1) \oplus  A_1(-1)^{\oplus 2}$	& $2^6$ 	& $1$ 	& $D_6 + D_4 +2 A_1$ 		& $\{ \mathbb{I} \}$			& 80\\
& $\quad \cong H\oplus D_8(-1) \oplus A_1(-1)^{\oplus 4}$ 			& 		&		& $D_8 + 4 A_1$ 			& $\{ \mathbb{I} \}$			& 40\\
& $\quad \cong H\oplus E_7(-1) \oplus A_1(-1)^{\oplus 5}$ 			&		& 		& $E_7 + 5 A_1$ 			& $\{ \mathbb{I} \}$			& 16\\
&	 													&		& 		& {\color{green}$D_6 + 6 A_1$}	& $\mathbb{Z}/2\mathbb{Z}$	& 16\\
& 														&		& 		& $2 D_4 + 4 A_1$			& $\mathbb{Z}/2\mathbb{Z}$	& 10\\
\hdashline
$14$ 	& $H\oplus D_4(-1)^{\oplus 2} \oplus  A_1(-1)^{\oplus 4}$ 		& $2^8$ 	& $1$ 	& $2 D_4 + 4 A_1$ 			& $\{ \mathbb{I} \}$			& 1920\\
& $\quad \cong H\oplus D_6(-1) \oplus  A_1(-1)^{\oplus 6}$ 			& 		& 		& $D_6 + 6 A_1$ 			& $\{ \mathbb{I} \}$			& 2304\\
& $\quad \cong \langle 2 \rangle \oplus \langle -2 \rangle\oplus D_4(-1)^{\oplus 3}$ &	&  		& $D_4 + 8 A_1$ 			& $\mathbb{Z}/2\mathbb{Z}$	& 1152\\
& 														& 		& 		& $12 A_1$ 				& $(\mathbb{Z}/2\mathbb{Z})^2$& 40 \\
\hline
$15$ 	& $H\oplus E_8(-1) \oplus A_1(-1)^{\oplus 5}$ 				& $2^5$ 	& $1$ 	& $E_8 + 5 A_1$ 			& $\{ \mathbb{I} \}$			& 1\\
& $\quad \cong H\oplus E_7(-1) \oplus D_4(-1) \oplus  A_1(-1)^{\oplus 2}$& 		& 		& $E_7 + D_4 + 2 A_1$		& $\{ \mathbb{I} \}$			& 10\\
& $\quad \cong H\oplus D_8(-1) \oplus D_4(-1) \oplus  A_1(-1)$			& 		& 		& $D_8 + D_4 + A_1$		& $\{ \mathbb{I} \}$			& 10\\
& $\quad \cong H\oplus D_6(-1)^{\oplus 2} \oplus  A_1(-1)$			& 		& 		& $2 D_6 + A_1$			& $\{ \mathbb{I} \}$			& 15\\
& 														& 		& 		& $E_7 + 6 A_1$ 			& $\mathbb{Z}/2\mathbb{Z}$	& 1\\
& 														& 		& 		& $D_8+ 5 A_1$ 			& $\mathbb{Z}/2\mathbb{Z}$	& 5\\
& 														& 		& 		& $D_6 + D_4 + 3 A_1$ 		& $\mathbb{Z}/2\mathbb{Z}$	& 10\\
\hdashline
$15$ 	& $H\oplus E_7(-1) \oplus A_1(-1)^{\oplus 6}$	 			& $2^7$ 	& $1$ 	& $E_7 + 6 A_1$ 			& $\{ \mathbb{I} \}$			& 72\\
& $\quad \cong H\oplus D_8(-1) \oplus  A_1(-1)^{\oplus 5}$			& 		& 		& $D_8 + 5 A_1$			& $\{ \mathbb{I} \}$			& 216\\
& $\quad \cong H\oplus D_6(-1) \oplus D_4(-1) \oplus  A_1(-1)^{\oplus 3}$& 		& 		& $D_6 + D_4 + 3 A_1$		& $\{ \mathbb{I} \}$			& 720\\
& $\quad \cong H\oplus D_4(-1)^{\oplus 3} \oplus  A_1(-1)$			& 		& 		& $3 D_4 + A_1$			& $\{ \mathbb{I} \}$			& 40\\
& 														& 		& 		& {\color{blue}$D_6 + 7 A_1$}	& $\mathbb{Z}/2\mathbb{Z}$	& 216\\
& 														& 		& 		& {\color{blue}$D_6 + 7 A_1$}	& $\mathbb{Z}/2\mathbb{Z}$	& 36\\
& 														& 		& 		& $2 D_4 + 5 A_1$ 			& $\mathbb{Z}/2\mathbb{Z}$	& 270\\
& 														& 		& 		& $D_4 + 9 A_1$ 			& $(\mathbb{Z}/2\mathbb{Z})^2$&40\\
\hline
$16$ 	& $H\oplus E_8(-1) \oplus D_4(-1) \oplus A_1(-1)^{\oplus 2}$ 	&  $2^4$ & $1$ 	& $E_8 + D_4 + 2 A_1$ 		& $\{ \mathbb{I} \}$			& 1\\
& $\quad \cong H\oplus E_7(-1) \oplus D_6(-1) \oplus A_1(-1)$ 			&		& 		& $E_7 + D_6 + A_1$ 		& $\{ \mathbb{I} \}$			& 6\\
& $\quad \cong H\oplus D_{12}(-1) \oplus A_1(-1)^{\oplus 2}$ 			&		& 		& $D_{12} + 2 A_1$			& $\{ \mathbb{I} \}$			& 3\\
& $\quad \cong H\oplus D_{10}(-1) \oplus D_4(-1)$ 					&		& 		& $D_{10} + D_4 $			& $\{ \mathbb{I} \}$			& 1\\
& $\quad \cong H\oplus D_8(-1) \oplus D_6(-1)$ 					&		& 		& $D_8 + D_6$ 			& $\{ \mathbb{I} \}$			& 3\\
&														& 		& 		& $E_7 + D_4 + 3 A_1$ 		& $\mathbb{Z}/2\mathbb{Z}$	& 1\\
&														& 		& 		& $D_{10}+ 4 A_1$ 			& $\mathbb{Z}/2\mathbb{Z}$	& 2\\
& 														& 		& 		& $D_8 + D_4 + 2 A_1$ 		& $\mathbb{Z}/2\mathbb{Z}$	& 3\\
& 														& 		& 		& $2 D_6 + 2 A_1$ 			& $\mathbb{Z}/2\mathbb{Z}$	& 3\\
\hline
\end{tabular}}
\captionsetup{justification=centering}
\caption{Jacobian elliptic K3 surfaces with $\mathrm{Aut}(\mathcal{X})= \infty$}
\label{tab:K3JEF2_a}
\end{table}
\begin{table}[!ht]
\scalemath{0.8}{
\begin{tabular}{c|clc|cc|c}
$\rho$ &  \multicolumn{1}{c}{$L= H \oplus K$}  & $\det{q_L}$ & $\delta$ & $K^{\text{root}}(-1)$ & $\mathpzc{W}$ & $|\mathpzc{A}|/|\mathpzc{B}|$\\
\hline
\hline
$16$ 	& $H\oplus E_8(-1) \oplus A_1(-1)^{\oplus 6}$ 				&  $2^6$ 	& $1$ 	& $E_8 + 6 A_1$ 				& $\{ \mathbb{I} \}$				& 2\\
& $\quad \cong H\oplus E_7(-1) \oplus D_4(-1) \oplus A_1(-1)^{\oplus 3}$ &		& 		& $E_7 + D_4 + 3 A_1$ 			& $\{ \mathbb{I} \}$				& 40\\
& $\quad \cong H\oplus D_{10}(-1) \oplus A_1(-1)^{\oplus 4}$ 			&		& 		& $D_{10} + 4 A_1$				& $\{ \mathbb{I} \}$				& 30\\
& $\quad \cong H\oplus D_8(-1) \oplus D_4(-1) \oplus A_1(-1)^{\oplus 2}$ &		& 		& $D_8 + D_4 + 2 A_1$ 			& $\{ \mathbb{I} \}$				& 60\\
& $\quad \cong H\oplus D_6(-1)^{\oplus 2} \oplus A_1(-1)^{\oplus 2}$ 	&		& 		& $2 D_6 + 2 A_1$ 				& $\{ \mathbb{I} \}$				& 90\\
& $\quad \cong H\oplus D_6(-1) \oplus D_4(-1)^{\oplus 2}$ 			&		& 		& $D_6 + 2 D_4$ 				& $\{ \mathbb{I} \}$				& 10\\
&														& 		& 		& $E_7 + 7 A_1$ 				& $\mathbb{Z}/2\mathbb{Z}$		& 6\\
&														&		&		& {\color{blue}$D_8 + 6 A_1$}		& $\mathbb{Z}/2\mathbb{Z}$		& 30\\
&														&		&		& {\color{blue}$D_8 + 6 A_1$}		& $\mathbb{Z}/2\mathbb{Z}$		& 1\\
& 														& 		& 		& {\color{blue}$D_6 + D_4 + 4 A_1$}	& $\mathbb{Z}/2\mathbb{Z}$ 		& 120\\
& 														& 		& 		& {\color{blue}$D_6 + D_4 + 4 A_1$}	& $\mathbb{Z}/2\mathbb{Z}$ 		& 15\\
& 														& 		& 		& $3  D_4 + 2 A_1$ 				& $\mathbb{Z}/2\mathbb{Z}$		& 15\\
& 														& 		& 		& $D_6 + 8 A_1$ 				& $(\mathbb{Z}/2\mathbb{Z})^2$	& 10\\
& (general double sextic)	 									& 		& 		& $2 D_4 + 6 A_1$ 				& $(\mathbb{Z}/2\mathbb{Z})^2$	& 15\\
\hline
$17$ 	& $H\oplus E_8(-1) \oplus D_6(-1) \oplus A_1(-1)$ 			&  $2^3$ 	& $1$ 	& $E_8 + D_6+A_1$ 			& $\{ \mathbb{I} \}$				& 1\\
& $\quad \cong H\oplus E_7(-1) \oplus E_7(-1) \oplus A_1(-1)$ 			&		& 		& $2 E_7 + A_1$ 				& $\{ \mathbb{I} \}$				& 1\\
& $\quad \cong H\oplus E_7(-1) \oplus D_8(-1)$ 					&		& 		& $E_7 + D_8$ 				& $\{ \mathbb{I} \}$				& 1\\
& $\quad \cong H\oplus D_{14}(-1) \oplus A_1(-1)$ 					&		& 		& $D_{14} + A_1$ 				& $\{ \mathbb{I} \}$				& 1\\
&														&		&		& $E_7 + D_6 + 2A_1$			& $\mathbb{Z}/2\mathbb{Z}$		& 1\\
& 														& 		& 		& $D_{12} + 3 A_1$ 				& $\mathbb{Z}/2\mathbb{Z}$		& 1\\
& 														& 		& 		& $D_{10} + D_4 + A_1$			& $\mathbb{Z}/2\mathbb{Z}$		& 1\\
& 														& 		& 		& $D_8 + D_6 + A_1$ 			& $\mathbb{Z}/2\mathbb{Z}$		& 2\\
\hdashline
$17$ 	& $H\oplus E_8(-1) \oplus D_4(-1) \oplus  A_1(-1)^{\oplus 3}$ 	& $2^5$ 	& $1$ 	& $E_8 + D_4 + 3 A_1$ 			& $\{ \mathbb{I} \}$				& 2\\
& $\quad \cong H\oplus E_7(-1) \oplus D_4(-1)^{\oplus 2}$ 			&		& 		& $E_7 + 2 D_4$ 				& $\{ \mathbb{I} \}$				& 1\\
& $\quad \cong H\oplus D_{12}(-1) \oplus A_1(-1)^{\oplus 3}$ 			&		& 		& $D_{12} + 3 A_1$ 				& $\{ \mathbb{I} \}$				& 6\\
& $\quad \cong H\oplus D_{10}(-1) \oplus D_4(-1) \oplus A_1(-1)$ 		&		& 		& $D_{10} + D_4 + A_1$ 			& $\{ \mathbb{I} \}$				& 6\\
& $\quad \cong H\oplus D_8(-1) \oplus D_6(-1) \oplus A_1(-1)$ 			&		& 		& $D_8 + D_6 + A_1$ 			& $\{ \mathbb{I} \}$				& 18\\
& 														& 		& 		& $E_7 + D_4 + 4 A_1$ 			& $\mathbb{Z}/2\mathbb{Z}$		& 6\\
& 														& 		& 		& $D_{10} + 5 A_1$ 				& $\mathbb{Z}/2\mathbb{Z}$		& 6\\
& 														& 		& 		& $D_8 + D_4 + 3 A_1$ 			& $\mathbb{Z}/2\mathbb{Z}$		& 18\\
& 														& 		& 		& {\color{blue}$2 D_6 + 3 A_1$} 	& $\mathbb{Z}/2\mathbb{Z}$		& 18\\
& 														& 		& 		& {\color{blue}$2 D_6 + 3 A_1$} 	& $\mathbb{Z}/2\mathbb{Z}$		& 6\\
& 														& 		& 		& $D_8 + 7 A_1$ 				& $(\mathbb{Z}/2\mathbb{Z})^2$ 	& 1\\
& 														& 		& 		& $D_6 + D_4 + 5 A_1$ 			& $(\mathbb{Z}/2\mathbb{Z})^2$ 	& 9\\
&  \multicolumn{3}{c|}{(twisted Legendre pencil -- 3 parameters)} 						& $3 D_4 + 3 A_1$ 				& $(\mathbb{Z}/2\mathbb{Z})^2$ 	& 2\\
\hline
$18$ 	& $H \oplus E_8(-1) \oplus D_8(-1)$ 						& $2^2$ 	& $0$ 	& $E_8 + D_8$ 				& $\{ \mathbb{I} \}$				& 1\\
& $\quad \cong H \oplus D_{16}(-1)$ 							& 		& 		& $D_{16}$					& $\{ \mathbb{I} \}$				& 1\\
& 														&		& 		& $2 E_7 + 2 A_1$				& $\mathbb{Z}/2\mathbb{Z}$		& 1\\
& 														&		& 		& $D_{12}+ D_4$				& $\mathbb{Z}/2\mathbb{Z}$		& 1\\
& 														&		& 		& $2 D_8$					& $\mathbb{Z}/2\mathbb{Z}$		& 2\\
\hdashline
$18$ 	& $H \oplus E_8(-1) \oplus E_7(-1) \oplus A_1(-1)$ 			& $2^2$ 	& $1$ 	& $E_8 + E_7 + A_1$ 			& $\{ \mathbb{I} \}$				& 1 \\
& 														&		& 		& $E_7 + D_8 + A_1$			& $\mathbb{Z}/2\mathbb{Z}$		& 1\\
& 														& 		& 		& $D_{14} + 2 A_1$				& $\mathbb{Z}/2\mathbb{Z}$		& 1\\
& 														&		& 		& $D_{10} + D_6$				& $\mathbb{Z}/2\mathbb{Z}$		& 1\\
\hdashline
$18$ 	& $H \oplus E_8(-1) \oplus D_4(-1)^{\oplus 2}$ 				& $2^4$ 	& $0$ 	& $E_8 + 2 D_4$ 				& $\{ \mathbb{I} \}$				& 1\\
& $\quad \cong H \oplus D_{12}(-1) \oplus D_4(-1)$ 					& 		& 		& $D_{12} + D_4$ 				& $\{ \mathbb{I} \}$				& 6\\
& $\quad \cong H \oplus D_8(-1) \oplus D_8(-1)$ 					& 		& 		& $2 D_8$ 					& $\{ \mathbb{I} \}$				& 9\\
& 														& 		& 		& $D_8 + 2 D_4$ 				& $\mathbb{Z}/2\mathbb{Z}$		& 9\\
& 														& 		& 		& {\color{green}$E_7 + D_6 + 3 A_1$}	& $\mathbb{Z}/2\mathbb{Z}$		& 6\\
& 														& 		& 		& $D_{10} + 6 A_1$ 				& $(\mathbb{Z}/2\mathbb{Z})^2$ 	& 1\\
& 														& 		& 		&  {\color{green}$2 D_6 + 4 A_1$ }	& $(\mathbb{Z}/2\mathbb{Z})^2$ 	& 9\\
&  \multicolumn{3}{c|}{(general Kummer $\mathrm{Kum}(E_1 \times E_2)$)	} 			& $4 D_4$ 					& $(\mathbb{Z}/2\mathbb{Z})^2$ 	& 2\\
\hline
\end{tabular}}
\captionsetup{justification=centering}
\caption{Jacobian elliptic K3 surfaces with $\mathrm{Aut}(\mathcal{X})= \infty$ (cont'd)}
\label{tab:K3JEF2_b}
\end{table}
\begin{table}[!ht]
\scalemath{0.85}{
\begin{tabular}{c|clc|cc|c}
$\rho$ &  \multicolumn{1}{c}{$L= H \oplus K$}  & $\det{q_L}$ & $\delta$ & $K^{\text{root}}(-1)$ & $\mathpzc{W}$ & $|\mathpzc{A}|/|\mathpzc{B}|$\\
\hline
\hline
$18$ 	& $H \oplus E_8(-1) \oplus D_6(-1) \oplus A_1(-1)^{\oplus 2}$ 	& $2^4$ 	& $1$ 	& $E_8 + D_6 + 2A_1$ 			& $\{ \mathbb{I} \}$					& 2\\
& $\quad \cong H \oplus E_7(-1) \oplus  E_7(-1) \oplus A_1(-1)^{\oplus 2}$& 		& 		& $2 E_7 + 2 A_1$ 				& $\{ \mathbb{I} \}$					& 2\\
& $\quad \cong H \oplus E_7(-1) \oplus  D_8(-1) \oplus A_1(-1)$ 		& 		& 		& $E_7 + D_8 + A_1$ 			& $\{ \mathbb{I} \}$					& 4\\
& $\quad \cong H \oplus D_{14}(-1) \oplus A_1(-1)^{\oplus 2}$ 			& 		& 		& $D_{14} + 2 A_1$ 				& $\{ \mathbb{I} \}$					& 2\\
& $\quad \cong H \oplus D_{10}(-1) \oplus D_6(-1)$ 					& 		& 		& $D_{10} + D_6$ 				& $\{ \mathbb{I} \}$					& 2\\
& 														& 		& 		& {\color{green}$E_7 + D_6 + 3 A_1$}	& $\mathbb{Z}/2\mathbb{Z}$			& 4\\
& 														& 		& 		& $E_7 + 2 D_4 + A_1$ 			& $\mathbb{Z}/2\mathbb{Z}$			& 1\\
& 														& 		& 		& $D_{12} + 4 A_1$ 				& $\mathbb{Z}/2\mathbb{Z}$			& 2\\
& 														& 		& 		& $D_{10}+ D_4 + 2 A_1$			& $\mathbb{Z}/2\mathbb{Z}$ 			& 4\\
& 														& 		& 		&  {\color{blue}$D_8 + D_6 + 2 A_1$} & $\mathbb{Z}/2\mathbb{Z}$			& 8\\
& 														& 		& 		&  {\color{blue}$D_8 + D_6 + 2 A_1$} & $\mathbb{Z}/2\mathbb{Z}$			& 2\\
& 														& 		& 		& $2 D_6 + D_4$ 				& $\mathbb{Z}/2\mathbb{Z}$			& 2\\
& 														& 		& 		&  {\color{green}$2 D_6 + 4 A_1$ }	& $(\mathbb{Z}/2\mathbb{Z})^2$ 		& 2\\
& 														& 		& 		& $D_8 + D_4 + 4 A_1$ 			& $(\mathbb{Z}/2\mathbb{Z})^2$ 		& 1\\
& \multicolumn{3}{c|}{(twisted Legendre pencil -- 2 parameters)} 						& $D_6 + 2 D_4 + 2 A_1$ 			& $(\mathbb{Z}/2\mathbb{Z})^2$ 		& 2\\
\hline
$19$ & $H \oplus E_8(-1) \oplus E_7(-1) \oplus A_1(-1)^{\oplus 2}$ 		& $2^3$ 	& $1$ 	& $E_8 + E_7 + 2 A_1$ 			& $\{ \mathbb{I} \}$					& 1\\
& $\quad \cong H \oplus E_8(-1) \oplus D_8(-1) \oplus A_1(-1)$ 		& 		& 		& $E_8 + D_8 + A_1$ 			& $\{ \mathbb{I} \}$					& 1\\
& $\quad \cong H \oplus E_7(-1) \oplus D_{10}(-1)$ 					& 		& 		& $E_7 + D_{10}$ 				& $\{ \mathbb{I} \}$					& 1\\
& $\quad \cong H \oplus D_{16}(-1) \oplus A_1(-1)$ 						& 		& 		& $D_{16} + A_1$ 				& $\{ \mathbb{I} \}$					& 1\\
& 														& 		& 		& $2 E_7 + 3 A_1$ 				& $\mathbb{Z}/2\mathbb{Z}$			& 1\\
& 														& 		& 		& $E_7 + D_8 + 2 A_1$ 			& $\mathbb{Z}/2\mathbb{Z}$			& 2\\
& 														& 		& 		& $E_7 + D_6 + D_4$ 			& $\mathbb{Z}/2\mathbb{Z}$			& 1\\
& 														& 		& 		& $D_{14}+ 3 A_1$ 				& $\mathbb{Z}/2\mathbb{Z}$			& 1\\
& 														& 		& 		& $D_{12}+ D_4 + A_1$ 			& $\mathbb{Z}/2\mathbb{Z}$			& 1\\
& 														& 		& 		& {\color{blue}$D_{10} + D_6 + A_1$} & $\mathbb{Z}/2\mathbb{Z}$			& 2\\
& 														& 		& 		& {\color{blue}$D_{10} + D_6 + A_1$} & $\mathbb{Z}/2\mathbb{Z}$			& 1\\
& 														& 		& 		& $2 D_8 + A_1$ 				& $\mathbb{Z}/2\mathbb{Z}$			& 2\\
& 														& 		& 		& $D_8 + D_6 + 3 A_1$ 			& $(\mathbb{Z}/2\mathbb{Z})^2$ 		& 1\\
& \multicolumn{3}{c|}{(twisted Legendre pencil -- 1 parameter)} 							& $2 D_6 + D_4 + A_1$ 			& $(\mathbb{Z}/2\mathbb{Z})^2$ 		& 1\\
\hline
$20$ 	& $H \oplus E_8(-1) \oplus E_8(-1) \oplus A_1(-1)^{\oplus 2}$ 	& $2^2$ 	& $1$ 	& $2 E_8 + 2 A_1$ 				& $\{ \mathbb{I} \}$					& 1\\
& $\quad \cong H \oplus E_8(-1) \oplus D_{10}(-1)$ 					& 		& 		& $E_8 + D_{10}$ 				& $\{ \mathbb{I} \}$					& 1\\
& $\quad \cong H \oplus D_{18}(-1)$ 							& 		& 		& $D_{18}$ 					& $\{ \mathbb{I} \}$					& 1\\
& 														& 		& 		& $2 E_7 + D_4$ 				& $\mathbb{Z}/2\mathbb{Z}$			& 1\\
& 														& 		& 		& $E_7 + D_{10} + A_1$ 			& $\mathbb{Z}/2\mathbb{Z}$			& 2\\
& 														& 		& 		& $D_{16} + 2 A_1$ 				& $\mathbb{Z}/2\mathbb{Z}$			& 1\\
& 														& 		& 		& $D_{12} + D_6$ 				& $\mathbb{Z}/2\mathbb{Z}$			& 1\\
& 														& 		& 		& $2 D_8 + 2 A_1$ 				& $(\mathbb{Z}/2\mathbb{Z})^2$ 		& 1\\
& 														& 		& 		& $3 D_6$ 					& $(\mathbb{Z}/2\mathbb{Z})^2$ 		& 1\\
\hline
\end{tabular}}
\captionsetup{justification=centering}
\caption{Jacobian elliptic K3 surfaces with $\mathrm{Aut}(\mathcal{X})= \infty$ (cont'd)}
\label{tab:K3JEF2_c}
\end{table}
\begin{table}[!ht]
\scalemath{0.8}{
\begin{tabular}{c|l|l}
$(\rho, \ell, \delta)$ & \multicolumn{1}{c|}{$K^{\text{root}}(-1)$}  &  \multicolumn{1}{c}{$\vec{v}_{\text{max}}$}  \\
 \hline
 $(10,6,0)$ & $8 A_1$ 		& $\left[{\frac{1}{2}}, {\frac{1}{2}}, {\frac{1}{2}}, {\frac{1}{2}}, {\frac{1}{2}}, {\frac{1}{2}}, {\frac{1}{2}}, {\frac{1}{2}}\right]$ \\[0.2em]
 $(11,7,1)$ &  $9 A_1$		& $\left[{\frac{1}{2}}, {\frac{1}{2}}, {\frac{1}{2}}, {\frac{1}{2}}, {\frac{1}{2}}, {\frac{1}{2}}, {\frac{1}{2}}, {\frac{1}{2}}, 0\right]$ \\[0.2em]
 $(12,6,1)$ & $6 A_1+ D_4$	& $\left[{\frac{1}{2}}, {\frac{1}{2}}, {\frac{1}{2}}, {\frac{1}{2}}, {\frac{1}{2}}, {\frac{1}{2}}, {\frac{1}{2}}, 0, {\frac{1}{2}}, 0\right]$ \\[0.2em]
 $(13,5,1)$ & $5 A_1+ D_6$	& $\left[{\frac{1}{2}}, {\frac{1}{2}}, {\frac{1}{2}}, {\frac{1}{2}}, {\frac{1}{2}}, {\frac{1}{2}}, 0, {\frac{1}{2}}, 0, 0, {\frac{1}{2}}\right]$ \\[0.2em]
 $(14,4,0)$ & $5 A_1+ E_7$	& $\left[{\frac{1}{2}}, {\frac{1}{2}}, {\frac{1}{2}}, {\frac{1}{2}}, {\frac{1}{2}}, 0, {\frac{1}{2}}, 0, 0, {\frac{1}{2}}, 0, {\frac{1}{2}}\right]$ \\[0.2em]
 $(14,4,1)$ & $4 A_1+ D_8$	& $\left[{\frac{1}{2}}, {\frac{1}{2}}, {\frac{1}{2}}, {\frac{1}{2}}, {\frac{1}{2}}, 0, {\frac{1}{2}}, 0, {\frac{1}{2}}, 0, {\frac{1}{2}}, 0\right]$ \\[0.2em]
 $(15,3,1)$ & $3 A_1+ D_{10}$	& $\left[{\frac{1}{2}}, {\frac{1}{2}}, {\frac{1}{2}}, {\frac{1}{2}}, 0, {\frac{1}{2}}, 0, {\frac{1}{2}}, 0, {\frac{1}{2}}, 0, 0, {\frac{1}{2}}\right] $ \\[0.2em]
 $(16,2,1)$ & $2 A_1+ D_{12}$	& $\left[{\frac{1}{2}}, {\frac{1}{2}}, {\frac{1}{2}}, 0, {\frac{1}{2}}, 0, {\frac{1}{2}}, 0, {\frac{1}{2}}, 0, {\frac{1}{2}}, 0, {\frac{1}{2}}, 0\right] $ \\[0.2em]
 $(17,1,0)$ & $A_1 + D_{14}$ 	& $\left[{\frac{1}{2}}, {\frac{1}{2}}, 0, {\frac{1}{2}}, 0, {\frac{1}{2}}, 0, {\frac{1}{2}}, 0, {\frac{1}{2}}, 0, {\frac{1}{2}}, 0, 0, {\frac{1}{2}}\right] $ \\[0.2em]
 $(18,0,0)$ & $D_{16}$		& $\left[{\frac{1}{2}}, 0, {\frac{1}{2}}, 0, {\frac{1}{2}}, 0, {\frac{1}{2}}, 0, {\frac{1}{2}}, 0, {\frac{1}{2}}, 0, {\frac{1}{2}}, 0, {\frac{1}{2}}, 0\right] $ \\[0.2em]
 $(19,1,1)$ & $A_1 + D_{16}$ 	& $\left[0, {\frac{1}{2}}, 0, {\frac{1}{2}}, 0, {\frac{1}{2}}, 0, {\frac{1}{2}}, 0, {\frac{1}{2}}, 0, {\frac{1}{2}}, 0, {\frac{1}{2}}, 0, {\frac{1}{2}}, 0\right] $  \\[0.2em]
 $(19,1,1)$ & $D_{10} + E_7$ 	& $\left[{\frac{1}{2}}, 0, {\frac{1}{2}}, 0, {\frac{1}{2}}, 0, {\frac{1}{2}}, 0, 0, {\frac{1}{2}}, 0, {\frac{1}{2}}, 0, 0, {\frac{1}{2}}, 0, {\frac{1}{2}}\right] $ \\[0.2em]
 \hline
\end{tabular}}
\captionsetup{justification=centering}
\caption{Vector generating a $\mathbb{Z}_2$-overlattice of $K^{\text{root}}$}
\label{tab:max_vects_Z2_finite}
\end{table}
\begin{table}[!ht]
\scalemath{0.8}{
\begin{tabular}{c|l|l}
$(\rho, \ell, \delta)$ & \multicolumn{1}{c|}{$K^{\text{root}}(-1)$}  &  \multicolumn{1}{c}{$\vec{v}_{\text{max}}$}  \\
 \hline
 $(12,8,1)$ & $10 A_1$ 				& $\left[{\frac{1}{2}}, {\frac{1}{2}}, {\frac{1}{2}}, {\frac{1}{2}}, {\frac{1}{2}}, {\frac{1}{2}}, {\frac{1}{2}}, {\frac{1}{2}}, 0, 0\right]$ \\[0.2em]
 $(13,7,1)$ & $7 A_1+D_4$ 			& $\left[{\frac{1}{2}}, {\frac{1}{2}}, {\frac{1}{2}}, {\frac{1}{2}}, {\frac{1}{2}}, {\frac{1}{2}}, 0, {\frac{1}{2}}, 0, {\frac{1}{2}}, 0\right]$ \\[0.2em]
 $(13,9,1)$ & $11 A_1$ 				& $\left[{\frac{1}{2}}, {\frac{1}{2}}, {\frac{1}{2}}, {\frac{1}{2}}, {\frac{1}{2}}, {\frac{1}{2}}, {\frac{1}{2}}, {\frac{1}{2}}, 0, 0, 0\right]$ \\[0.2em]
 $(14,6,0)$ & {\color{green}$6 A_1 + D_6$} & $\left[{\frac{1}{2}}, {\frac{1}{2}}, {\frac{1}{2}}, {\frac{1}{2}}, {\frac{1}{2}}, {\frac{1}{2}}, 0, 0, 0, 0, {\frac{1}{2}}, {\frac{1}{2}}\right]$ \\[0.2em]
 $(14,6,1)$ & {\color{green}$6 A_1 + D_6$} & $\left[{\frac{1}{2}}, {\frac{1}{2}}, {\frac{1}{2}}, {\frac{1}{2}}, {\frac{1}{2}}, 0, {\frac{1}{2}}, 0, {\frac{1}{2}}, 0, 0, {\frac{1}{2}}\right]$ \\[0.2em]
 $(14,6,1)$ &$4 A_1 + 2 D_4$			& $\left[{\frac{1}{2}}, {\frac{1}{2}}, {\frac{1}{2}}, {\frac{1}{2}}, {\frac{1}{2}}, 0, {\frac{1}{2}}, 0, {\frac{1}{2}}, 0, {\frac{1}{2}}, 0\right]$ \\[0.2em]
 $(14,8,1)$ &$8 A_1 + D_4$			& $\left[{\frac{1}{2}}, {\frac{1}{2}}, {\frac{1}{2}}, {\frac{1}{2}}, {\frac{1}{2}}, {\frac{1}{2}}, 0, 0, {\frac{1}{2}}, 0, {\frac{1}{2}}, 0\right]$ \\[0.2em]
 $(15,5,1)$ &$6 A_1 + E_7$			& $\left[{\frac{1}{2}}, {\frac{1}{2}}, {\frac{1}{2}}, {\frac{1}{2}}, {\frac{1}{2}}, 0, 0, {\frac{1}{2}}, 0, 0, {\frac{1}{2}}, 0, {\frac{1}{2}}\right]$ \\[0.2em]
 $(15,5,1)$ &$5 A_1 + D_8$			& $\left[{\frac{1}{2}}, {\frac{1}{2}}, {\frac{1}{2}}, {\frac{1}{2}}, 0, {\frac{1}{2}}, 0, {\frac{1}{2}}, 0, {\frac{1}{2}}, 0, {\frac{1}{2}}, 0\right]$ \\[0.2em]
 $(15,5,1)$ &$3 A_1 + D_4 + D_6$		& $\left[{\frac{1}{2}}, {\frac{1}{2}}, {\frac{1}{2}}, {\frac{1}{2}}, 0, {\frac{1}{2}}, 0, {\frac{1}{2}}, 0, {\frac{1}{2}}, 0, 0, {\frac{1}{2}}\right]$ \\[0.2em]
 $(15,7,1)$ &{\color{blue}$7 A_1 + D_6$}	& $\left[{\frac{1}{2}}, {\frac{1}{2}}, {\frac{1}{2}}, {\frac{1}{2}}, {\frac{1}{2}}, 0, 0, {\frac{1}{2}}, 0, {\frac{1}{2}}, 0, 0, {\frac{1}{2}}\right]$ \\[0.2em]
 $(15,7,1)$ &{\color{blue}$7 A_1 + D_6$}	& $\left[{\frac{1}{2}}, {\frac{1}{2}}, {\frac{1}{2}}, {\frac{1}{2}}, {\frac{1}{2}}, {\frac{1}{2}}, 0, 0, 0, 0, 0, {\frac{1}{2}}, {\frac{1}{2}}\right]$ \\[0.2em]
 $(15,7,1)$ &$5 A_1 + 2 D_4$			& $\left[{\frac{1}{2}}, {\frac{1}{2}}, {\frac{1}{2}}, {\frac{1}{2}}, 0, {\frac{1}{2}}, 0, {\frac{1}{2}}, 0, {\frac{1}{2}}, 0, {\frac{1}{2}}, 0\right]$ \\[0.2em]
 $(16,4,1)$ &$3 A_1 + D_4 + E_7$		& $\left[{\frac{1}{2}}, {\frac{1}{2}}, {\frac{1}{2}}, {\frac{1}{2}}, 0, {\frac{1}{2}}, 0, 0, {\frac{1}{2}}, 0, 0, {\frac{1}{2}}, 0, {\frac{1}{2}}\right]$ \\[0.2em]	
 $(16,4,1)$ &$4 A_1 + D_{10}$			& $\left[{\frac{1}{2}}, {\frac{1}{2}}, {\frac{1}{2}}, 0, {\frac{1}{2}}, 0, {\frac{1}{2}}, 0, {\frac{1}{2}}, 0, {\frac{1}{2}}, 0, 0, {\frac{1}{2}}\right]$ \\[0.2em]
 $(16,4,1)$ &$2 A_1 + D_4 +D_8$		& $\left[{\frac{1}{2}}, {\frac{1}{2}}, {\frac{1}{2}}, 0, {\frac{1}{2}}, 0, {\frac{1}{2}}, 0, {\frac{1}{2}}, 0, {\frac{1}{2}}, 0, {\frac{1}{2}}, 0\right]$ \\[0.2em]
 $(16,4,1)$ &$2 A_1 + 2 D_6$			& $\left[{\frac{1}{2}}, {\frac{1}{2}}, {\frac{1}{2}}, 0, {\frac{1}{2}}, 0, 0, {\frac{1}{2}}, {\frac{1}{2}}, 0, {\frac{1}{2}}, 0, 0, {\frac{1}{2}}\right]$ \\[0.2em]
 $(16,6,1)$ &$7 A_1 + E_7$			& $\left[{\frac{1}{2}}, {\frac{1}{2}}, {\frac{1}{2}}, {\frac{1}{2}}, {\frac{1}{2}}, 0, 0, 0, {\frac{1}{2}}, 0, 0, {\frac{1}{2}}, 0, {\frac{1}{2}}\right]$ \\[0.2em]
 $(16,6,1)$ &{\color{blue}$6 A_1 + D_8$}	& $\left[{\frac{1}{2}}, {\frac{1}{2}}, {\frac{1}{2}}, {\frac{1}{2}}, 0, 0, {\frac{1}{2}}, 0, {\frac{1}{2}}, 0, {\frac{1}{2}}, 0, {\frac{1}{2}}, 0\right]$ \\[0.2em]
 $(16,6,1)$ &{\color{blue}$6 A_1 + D_8$}	& $\left[{\frac{1}{2}}, {\frac{1}{2}}, {\frac{1}{2}}, {\frac{1}{2}}, {\frac{1}{2}}, {\frac{1}{2}}, 0, 0, 0, 0, 0, 0, {\frac{1}{2}}, {\frac{1}{2}}\right]$ \\[0.2em]
 $(16,6,1)$ &{\color{blue}$4 A_1 + D_4 + D_6$}& $\left[{\frac{1}{2}}, {\frac{1}{2}}, {\frac{1}{2}}, 0, {\frac{1}{2}}, 0, {\frac{1}{2}}, 0, {\frac{1}{2}}, 0, {\frac{1}{2}}, 0, 0, {\frac{1}{2}}\right]$ \\[0.2em]
 $(16,6,1)$ &{\color{blue}$4 A_1 + D_4 + D_6$}& $\left[{\frac{1}{2}}, {\frac{1}{2}}, {\frac{1}{2}}, {\frac{1}{2}}, {\frac{1}{2}}, 0, {\frac{1}{2}}, 0, 0, 0, 0, 0, {\frac{1}{2}}, {\frac{1}{2}}\right]$ \\[0.2em]
 $(16,6,1)$ &$2 A_1 + 3 D_4 $			& $\left[{\frac{1}{2}}, {\frac{1}{2}}, {\frac{1}{2}}, 0, {\frac{1}{2}}, 0, {\frac{1}{2}}, 0, {\frac{1}{2}}, 0, {\frac{1}{2}}, 0, {\frac{1}{2}}, 0\right]$ \\[0.2em]
 \hline
  \end{tabular}}
\captionsetup{justification=centering}
\caption{Vector generating a $\mathbb{Z}_2$-overlattice  of $K^{\text{root}}$}
\label{tab:max_vects_Z2_infinite_a}
\end{table}
\begin{table}[!ht]
\scalemath{0.8}{
\begin{tabular}{c|l|l}
$(\rho, \ell, \delta)$ & \multicolumn{1}{c|}{$K^{\text{root}}(-1)$}  &  \multicolumn{1}{c}{$\vec{v}_{\text{max}}$}  \\
 \hline
  $(17,3,1)$ &$2 A_1 + D_{6} + E_7$		& $\left[{\frac{1}{2}}, {\frac{1}{2}}, {\frac{1}{2}}, 0, {\frac{1}{2}}, 0, 0, {\frac{1}{2}}, 0, {\frac{1}{2}}, 0, 0, {\frac{1}{2}}, 0,  {\frac{1}{2}}\right]$ \\[0.2em]
 $(17,3,1)$ &$3 A_1 + D_{12}$			& $\left[{\frac{1}{2}}, {\frac{1}{2}}, 0, {\frac{1}{2}}, 0, {\frac{1}{2}}, 0, {\frac{1}{2}}, 0, {\frac{1}{2}}, 0, {\frac{1}{2}}, 0, {\frac{1}{2}}, 0\right]$ \\[0.2em]
 $(17,3,1)$ &$A_1 + D_4 +D_{10}$		& $\left[{\frac{1}{2}}, {\frac{1}{2}}, 0, {\frac{1}{2}}, 0, {\frac{1}{2}}, 0, {\frac{1}{2}}, 0, {\frac{1}{2}}, 0, {\frac{1}{2}}, 0, 0, {\frac{1}{2}}\right]$ \\[0.2em]
 $(17,3,1)$ &$ A_1 + D_6 +D_8$		& $\left[{\frac{1}{2}}, {\frac{1}{2}}, 0, {\frac{1}{2}}, 0, 0, {\frac{1}{2}}, {\frac{1}{2}}, 0, {\frac{1}{2}}, 0, {\frac{1}{2}}, 0, {\frac{1}{2}}, 0\right]$ \\[0.2em]
 $(17,5,1)$ &$4A_1 + D_4 +E_7$		& $\left[{\frac{1}{2}}, {\frac{1}{2}}, {\frac{1}{2}}, 0, {\frac{1}{2}}, 0, {\frac{1}{2}}, 0, 0, {\frac{1}{2}}, 0, 0, {\frac{1}{2}}, 0, {\frac{1}{2}}\right]$ \\[0.2em]
 $(17,5,1)$ &$5A_1 + D_{10}$			& $\left[{\frac{1}{2}}, {\frac{1}{2}}, {\frac{1}{2}}, 0, 0, {\frac{1}{2}}, 0, {\frac{1}{2}}, 0, {\frac{1}{2}}, 0, {\frac{1}{2}}, 0, 0, {\frac{1}{2}}\right]$ \\[0.2em]
 $(17,5,1)$ &$3 A_1 + D_4 + D_8$		& $\left[{\frac{1}{2}}, {\frac{1}{2}}, 0, {\frac{1}{2}}, 0, {\frac{1}{2}}, 0, {\frac{1}{2}}, 0, {\frac{1}{2}}, 0, {\frac{1}{2}}, 0, {\frac{1}{2}}, 0\right]$ \\[0.2em]
 $(17,5,1)$ &{\color{blue}$3 A_1 + 2 D_6$	}&$\left[{\frac{1}{2}}, {\frac{1}{2}}, 0, {\frac{1}{2}}, 0, {\frac{1}{2}}, 0, 0, {\frac{1}{2}}, {\frac{1}{2}}, 0, {\frac{1}{2}}, 0, 0, {\frac{1}{2}}\right]$ \\[0.2em]
 $(17,5,1)$ &{\color{blue}$3 A_1 + 2 D_6$	}&$\left[{\frac{1}{2}}, {\frac{1}{2}}, {\frac{1}{2}}, {\frac{1}{2}}, 0, {\frac{1}{2}}, 0, 0, {\frac{1}{2}}, 0, 0, 0, 0, {\frac{1}{2}}, {\frac{1}{2}}\right]$ \\[0.2em]
 $(18,2,0)$ &$2 A_1 + 2 E_7$			& $\left[{\frac{1}{2}}, {\frac{1}{2}}, 0, {\frac{1}{2}}, 0, 0, {\frac{1}{2}}, 0, {\frac{1}{2}}, 0, {\frac{1}{2}}, 0, 0, {\frac{1}{2}}, 0, {\frac{1}{2}}\right]$ \\[0.2em]
 $(18,2,0)$ &$D_4 + D_{12}$			& $\left[{\frac{1}{2}}, 0, {\frac{1}{2}}, 0, {\frac{1}{2}}, 0, {\frac{1}{2}}, 0, {\frac{1}{2}}, 0, {\frac{1}{2}}, 0, {\frac{1}{2}}, 0, {\frac{1}{2}}, 0\right]$ \\[0.2em]
 $(18,2,0)$ &$2 D_8$				& $\left[{\frac{1}{2}}, 0, {\frac{1}{2}}, 0, {\frac{1}{2}}, 0, {\frac{1}{2}}, 0, {\frac{1}{2}}, 0, {\frac{1}{2}}, 0, {\frac{1}{2}}, 0, {\frac{1}{2}}, 0\right]$ \\[0.2em]
 $(18,2,1)$ &$A_1 + D_8 +E_7$		& $\left[{\frac{1}{2}}, {\frac{1}{2}}, 0, {\frac{1}{2}}, 0, {\frac{1}{2}}, 0, {\frac{1}{2}}, 0, 0, {\frac{1}{2}}, 0, 0, {\frac{1}{2}}, 0, {\frac{1}{2}}\right]$ \\[0.2em]
 $(18,2,1)$ &$2A_1 + D_{14}$			& $\left[{\frac{1}{2}}, 0, {\frac{1}{2}}, 0, {\frac{1}{2}}, 0, {\frac{1}{2}}, 0, {\frac{1}{2}}, 0, {\frac{1}{2}}, 0, {\frac{1}{2}}, 0, 0, {\frac{1}{2}}\right]$ \\[0.2em]
 $(18,2,1)$ &$D_6 + D_{10}$			& $\left[{\frac{1}{2}}, 0, {\frac{1}{2}}, 0, 0, {\frac{1}{2}}, {\frac{1}{2}}, 0, {\frac{1}{2}}, 0, {\frac{1}{2}}, 0, {\frac{1}{2}}, 0, 0, {\frac{1}{2}}\right]$ \\[0.2em]
 $(18,4,0)$ &$2 D_4 + D_8$			& $\left[{\frac{1}{2}}, 0, {\frac{1}{2}}, 0, {\frac{1}{2}}, 0, {\frac{1}{2}}, 0, {\frac{1}{2}}, 0, {\frac{1}{2}}, 0, {\frac{1}{2}}, 0, {\frac{1}{2}}, 0\right]$ \\[0.2em]
 $(18,4,0)$ &{\color{green}$3 A_1 + D_6 +E_7$}&$\left[{\frac{1}{2}}, {\frac{1}{2}}, {\frac{1}{2}}, 0, 0, 0, 0, {\frac{1}{2}}, {\frac{1}{2}}, 0, {\frac{1}{2}}, 0, 0, {\frac{1}{2}}, 0, {\frac{1}{2}}\right]$ \\[0.2em]
 $(18,4,1)$ &{\color{green}$3 A_1 + D_6 +E_7$}&$\left[{\frac{1}{2}}, {\frac{1}{2}}, 0, {\frac{1}{2}}, 0, {\frac{1}{2}}, 0, 0, {\frac{1}{2}}, 0, {\frac{1}{2}}, 0, 0, {\frac{1}{2}}, 0, {\frac{1}{2}}\right]$ \\[0.2em]
 $(18,4,1)$ &$A_1 + 2 D_4 +E_7$		&$\left[{\frac{1}{2}}, {\frac{1}{2}}, 0, {\frac{1}{2}}, 0, {\frac{1}{2}}, 0, {\frac{1}{2}}, 0, 0, {\frac{1}{2}}, 0, 0, {\frac{1}{2}}, 0, {\frac{1}{2}}\right]$ \\[0.2em]
 $(18,4,1)$ &$4 A_1 + D_{12}$			&$\left[{\frac{1}{2}}, {\frac{1}{2}}, 0, 0, {\frac{1}{2}}, 0, {\frac{1}{2}}, 0, {\frac{1}{2}}, 0, {\frac{1}{2}}, 0, {\frac{1}{2}}, 0, {\frac{1}{2}}, 0\right]$ \\[0.2em]
 $(18,4,1)$ &$2 A_1 + D_4 + D_{10}$	&$\left[{\frac{1}{2}}, 0, {\frac{1}{2}}, 0, {\frac{1}{2}}, 0, {\frac{1}{2}}, 0, {\frac{1}{2}}, 0, {\frac{1}{2}}, 0, {\frac{1}{2}}, 0, 0, {\frac{1}{2}}\right]$ \\[0.2em]
 $(18,4,1)$ &{\color{blue}$2 A_1 + D_6 + D_8$}&$\left[{\frac{1}{2}}, 0, {\frac{1}{2}}, 0, {\frac{1}{2}}, 0, 0, {\frac{1}{2}}, {\frac{1}{2}}, 0, {\frac{1}{2}}, 0, {\frac{1}{2}}, 0, {\frac{1}{2}}, 0\right]$ \\[0.2em]
 $(18,4,1)$ &{\color{blue}$2 A_1 + D_6 + D_8$}&$\left[{\frac{1}{2}}, {\frac{1}{2}}, 0, 0, 0, 0, {\frac{1}{2}}, {\frac{1}{2}}, {\frac{1}{2}}, 0, {\frac{1}{2}}, 0, {\frac{1}{2}}, 0, {\frac{1}{2}}, 0\right]$ \\[0.2em]
 $(18,4,1)$ &$D_4+2D_6$			&$\left[{\frac{1}{2}}, 0, {\frac{1}{2}}, 0, {\frac{1}{2}}, 0, {\frac{1}{2}}, 0, 0, {\frac{1}{2}}, {\frac{1}{2}}, 0, {\frac{1}{2}}, 0, 0, {\frac{1}{2}}\right]$ \\[0.2em]
 $(19,3,1)$ &$3 A_1 + 2 E_7$			&$\left[{\frac{1}{2}}, {\frac{1}{2}}, 0, 0, {\frac{1}{2}}, 0, 0, {\frac{1}{2}}, 0, {\frac{1}{2}}, 0, {\frac{1}{2}}, 0, 0, {\frac{1}{2}}, 0, {\frac{1}{2}}\right]$ \\[0.2em]
 $(19,3,1)$ &$2 A_1 + D_8 + E_7$		&$\left[{\frac{1}{2}}, 0, {\frac{1}{2}}, 0, {\frac{1}{2}}, 0, {\frac{1}{2}}, 0, {\frac{1}{2}}, 0, 0, {\frac{1}{2}}, 0, 0, {\frac{1}{2}}, 0, {\frac{1}{2}}\right]$ \\[0.2em]
 $(19,3,1)$ &$D_4 + D_6 +E_7$		&$\left[{\frac{1}{2}}, 0, {\frac{1}{2}}, 0, {\frac{1}{2}}, 0, {\frac{1}{2}}, 0, 0, {\frac{1}{2}}, 0, {\frac{1}{2}}, 0, 0, {\frac{1}{2}}, 0, {\frac{1}{2}}\right]$ \\[0.2em]
 $(19,3,1)$ &$A_1 + D_4 + D_{12}$		&$\left[0, {\frac{1}{2}}, 0, {\frac{1}{2}}, 0, {\frac{1}{2}}, 0, {\frac{1}{2}}, 0, {\frac{1}{2}}, 0, {\frac{1}{2}}, 0, {\frac{1}{2}}, 0, {\frac{1}{2}}, 0\right]$ \\[0.2em]
 $(19,3,1)$ &{\color{blue}$A_1 + D_6 + D_{10}$}&$\left[0, {\frac{1}{2}}, 0, {\frac{1}{2}}, 0, 0, {\frac{1}{2}}, {\frac{1}{2}}, 0, {\frac{1}{2}}, 0, {\frac{1}{2}}, 0, {\frac{1}{2}}, 0, 0, {\frac{1}{2}}\right]$ \\[0.2em]
 $(19,3,1)$ &{\color{blue}$A_1 + D_6 + D_{10}$}&$\left[{\frac{1}{2}}, 0, 0, 0, 0, {\frac{1}{2}}, {\frac{1}{2}}, {\frac{1}{2}}, 0, {\frac{1}{2}}, 0, {\frac{1}{2}}, 0, {\frac{1}{2}}, 0, 0, {\frac{1}{2}}\right]$ \\[0.2em]
 $(19,3,1)$ &$A_1 + 2 D_8$			&$\left[0, {\frac{1}{2}}, 0, {\frac{1}{2}}, 0, {\frac{1}{2}}, 0, {\frac{1}{2}}, 0, {\frac{1}{2}}, 0, {\frac{1}{2}}, 0, {\frac{1}{2}}, 0, {\frac{1}{2}}, 0\right]$ \\[0.2em]
 $(20,2,1)$ &$D_4 + 2 E_7$			&$\left[{\frac{1}{2}}, 0, {\frac{1}{2}}, 0, 0, {\frac{1}{2}}, 0, 0, {\frac{1}{2}}, 0, {\frac{1}{2}}, 0, {\frac{1}{2}}, 0, 0, {\frac{1}{2}}, 0, {\frac{1}{2}}\right]$ \\[0.2em]
 $(20,2,1)$ &$A_1 + D_{10} + E_7$		&$\left[0, {\frac{1}{2}}, 0, {\frac{1}{2}}, 0, {\frac{1}{2}}, 0, {\frac{1}{2}}, 0, 0, {\frac{1}{2}}, 0, {\frac{1}{2}}, 0, 0, {\frac{1}{2}}, 0, {\frac{1}{2}}\right]$ \\[0.2em]
 $(20,2,1)$ &$2A_1 + D_{16}$			&$\left[0, 0, {\frac{1}{2}}, 0, {\frac{1}{2}}, 0, {\frac{1}{2}}, 0, {\frac{1}{2}}, 0, {\frac{1}{2}}, 0, {\frac{1}{2}}, 0, {\frac{1}{2}}, 0, {\frac{1}{2}}, 0\right]$ \\[0.2em]
 $(20,2,1)$ &$D_6 + D_{12}$			&$\left[0, 0, 0, 0, {\frac{1}{2}}, {\frac{1}{2}}, {\frac{1}{2}}, 0, {\frac{1}{2}}, 0, {\frac{1}{2}}, 0, {\frac{1}{2}}, 0, {\frac{1}{2}}, 0, {\frac{1}{2}}, 0\right]$ \\[0.2em]
\hline 
\end{tabular}}
\captionsetup{justification=centering}
\caption{Vector generating a $\mathbb{Z}_2$-overlattice of $K^{\text{root}}$}
\label{tab:max_vects_Z2_infinite_b}
\end{table} 
\begin{table}[!ht]
\scalemath{0.75}{
\begin{tabular}{c|l|l}
$(\rho, \ell, \delta)$ & \multicolumn{1}{c|}{$K^{\text{root}}(-1)$}  &  \multicolumn{1}{c}{$\vec{v}_{\text{max}, 1}, \vec{v}_{\text{max}, 2}$}  \\
 \hline
 $(14,8,1)$ &$12 A_1$			& $\left[{\frac{1}{2}}, {\frac{1}{2}}, {\frac{1}{2}}, {\frac{1}{2}}, 0, 0, 0, 0, {\frac{1}{2}}, {\frac{1}{2}}, {\frac{1}{2}}, {\frac{1}{2}}\right], 
 \left[0, 0, 0, 0, {\frac{1}{2}}, {\frac{1}{2}}, {\frac{1}{2}}, {\frac{1}{2}}, {\frac{1}{2}}, {\frac{1}{2}}, {\frac{1}{2}}, {\frac{1}{2}}\right]$ \\[0.2em]
 $(15,7,1)$ &$9 A_1+ D_4$		& $\left[{\frac{1}{2}}, {\frac{1}{2}}, {\frac{1}{2}}, 0, 0, 0, {\frac{1}{2}}, {\frac{1}{2}}, {\frac{1}{2}}, {\frac{1}{2}}, 0, {\frac{1}{2}}, 0\right],
\left[0, 0, 0, {\frac{1}{2}}, {\frac{1}{2}}, {\frac{1}{2}}, {\frac{1}{2}}, {\frac{1}{2}}, {\frac{1}{2}}, {\frac{1}{2}}, 0, 0, {\frac{1}{2}}\right]$ \\[0.2em]
$(16,6,1)$ &$8 A_1+ D_6$		& $\left[{\frac{1}{2}}, {\frac{1}{2}}, {\frac{1}{2}}, 0, 0, 0, {\frac{1}{2}}, {\frac{1}{2}}, {\frac{1}{2}}, 0, {\frac{1}{2}}, 0, 0, {\frac{1}{2}}\right],
\left[0, 0, 0, {\frac{1}{2}}, {\frac{1}{2}}, {\frac{1}{2}}, {\frac{1}{2}}, {\frac{1}{2}}, {\frac{1}{2}}, 0, {\frac{1}{2}}, 0, {\frac{1}{2}}, 0\right]$ \\[0.2em]
$(16,6,1)$ &$6 A_1+ 2 D_4$		& $\left[{\frac{1}{2}}, {\frac{1}{2}}, 0, 0, {\frac{1}{2}}, {\frac{1}{2}}, {\frac{1}{2}}, 0, {\frac{1}{2}}, 0, {\frac{1}{2}}, 0, {\frac{1}{2}}, 0\right],
\left[0, 0, {\frac{1}{2}}, {\frac{1}{2}}, {\frac{1}{2}}, {\frac{1}{2}}, {\frac{1}{2}}, 0, 0, {\frac{1}{2}}, {\frac{1}{2}}, 0, 0, {\frac{1}{2}}\right]$  \\[0.2em]
$(17,5,1)$ &$7 A_1+ D_8$		& $\left[{\frac{1}{2}}, {\frac{1}{2}}, {\frac{1}{2}}, 0, 0, 0, {\frac{1}{2}}, {\frac{1}{2}}, 0, {\frac{1}{2}}, 0, {\frac{1}{2}}, 0, {\frac{1}{2}}, 0\right],
\left[0, 0, 0, {\frac{1}{2}}, {\frac{1}{2}}, {\frac{1}{2}}, {\frac{1}{2}}, {\frac{1}{2}}, 0, {\frac{1}{2}}, 0, {\frac{1}{2}}, 0, 0, {\frac{1}{2}}\right]$  \\[0.2em]
$(17,5,1)$ &$5 A_1+ D_4 + D_6$	& $\left[{\frac{1}{2}}, {\frac{1}{2}}, 0, 0, {\frac{1}{2}}, {\frac{1}{2}}, 0, {\frac{1}{2}}, 0, {\frac{1}{2}}, 0, {\frac{1}{2}}, 0, 0, {\frac{1}{2}}\right],
\left[0, 0, {\frac{1}{2}}, {\frac{1}{2}}, {\frac{1}{2}}, {\frac{1}{2}}, 0, 0, {\frac{1}{2}}, {\frac{1}{2}}, 0, {\frac{1}{2}}, 0, {\frac{1}{2}}, 0\right]$  \\[0.2em]
$(17,5,1)$ &$3 A_1+ 3 D_4$		& $\left[{\frac{1}{2}}, 0, {\frac{1}{2}}, {\frac{1}{2}}, 0, {\frac{1}{2}}, 0, {\frac{1}{2}}, 0, {\frac{1}{2}}, 0, {\frac{1}{2}}, 0, {\frac{1}{2}}, 0\right],
\left[0, {\frac{1}{2}}, {\frac{1}{2}}, {\frac{1}{2}}, 0, 0, {\frac{1}{2}}, {\frac{1}{2}}, 0, 0, {\frac{1}{2}}, {\frac{1}{2}}, 0, 0, {\frac{1}{2}}\right]$  \\[0.2em]
$(18,4,0)$ &$6 A_1+ D_{10}$		& $\left[{\frac{1}{2}}, {\frac{1}{2}}, {\frac{1}{2}}, 0, 0, 0, {\frac{1}{2}}, 0, {\frac{1}{2}}, 0, {\frac{1}{2}}, 0, {\frac{1}{2}}, 0, 0, {\frac{1}{2}}\right],
\left[0, 0, 0, {\frac{1}{2}}, {\frac{1}{2}}, {\frac{1}{2}}, {\frac{1}{2}}, 0, {\frac{1}{2}}, 0, {\frac{1}{2}}, 0, {\frac{1}{2}}, 0, {\frac{1}{2}}, 0\right]$  \\[0.2em]
$(18,4,0)$ &$4 D_4$				& $\left[{\frac{1}{2}}, 0, {\frac{1}{2}}, 0, {\frac{1}{2}}, 0, {\frac{1}{2}}, 0, {\frac{1}{2}}, 0, {\frac{1}{2}}, 0, {\frac{1}{2}}, 0, {\frac{1}{2}}, 0\right],
\left[{\frac{1}{2}}, 0, 0, {\frac{1}{2}}, {\frac{1}{2}}, 0, 0, {\frac{1}{2}}, {\frac{1}{2}}, 0, 0, {\frac{1}{2}}, {\frac{1}{2}}, 0, 0, {\frac{1}{2}}\right]$  \\[0.2em]
$(18,4,0)$ & {\color{green}$4 A_1+ 2 D_6$} & $\left[{\frac{1}{2}}, {\frac{1}{2}}, 0, 0, {\frac{1}{2}}, 0, {\frac{1}{2}}, 0, 0, {\frac{1}{2}}, {\frac{1}{2}}, 0, {\frac{1}{2}}, 0, 0, {\frac{1}{2}}\right],
\left[0, 0, {\frac{1}{2}}, {\frac{1}{2}}, {\frac{1}{2}}, 0, {\frac{1}{2}}, 0, {\frac{1}{2}}, 0, {\frac{1}{2}}, 0, {\frac{1}{2}}, 0, {\frac{1}{2}}, 0\right]$  \\[0.2em]
$(18,4,1)$ &  {\color{green}$4 A_1+ 2 D_6$	}& $\left[{\frac{1}{2}}, {\frac{1}{2}}, 0, {\frac{1}{2}}, {\frac{1}{2}}, 0, {\frac{1}{2}}, 0, 0, {\frac{1}{2}}, 0, 0, 0, 0, {\frac{1}{2}}, {\frac{1}{2}}\right],
\left[0, 0, {\frac{1}{2}}, {\frac{1}{2}}, {\frac{1}{2}}, 0, {\frac{1}{2}}, 0, {\frac{1}{2}}, 0, {\frac{1}{2}}, 0, {\frac{1}{2}}, 0, {\frac{1}{2}}, 0\right]$\\[0.2em]
$(18,4,1)$ &$4 A_1 + D_4 + D_8$	& $\left[{\frac{1}{2}}, {\frac{1}{2}}, 0, 0, {\frac{1}{2}}, 0, {\frac{1}{2}}, 0, {\frac{1}{2}}, 0, {\frac{1}{2}}, 0, {\frac{1}{2}}, 0, {\frac{1}{2}}, 0\right],
\left[0, 0, {\frac{1}{2}}, {\frac{1}{2}}, {\frac{1}{2}}, 0, 0, {\frac{1}{2}}, {\frac{1}{2}}, 0, {\frac{1}{2}}, 0, {\frac{1}{2}}, 0, 0, {\frac{1}{2}}\right]$  \\[0.2em]
$(18,4,1)$ &$2 A_1 + 2 D_4 + D_6$	& $\left[{\frac{1}{2}}, 0, {\frac{1}{2}}, 0, {\frac{1}{2}}, 0, {\frac{1}{2}}, 0, {\frac{1}{2}}, 0, {\frac{1}{2}}, 0, {\frac{1}{2}}, 0, 0, {\frac{1}{2}}\right],
\left[0, {\frac{1}{2}}, {\frac{1}{2}}, 0, 0, {\frac{1}{2}}, {\frac{1}{2}}, 0, 0, {\frac{1}{2}}, {\frac{1}{2}}, 0, {\frac{1}{2}}, 0, {\frac{1}{2}}, 0\right]$  \\[0.2em]
$(19,3,1)$ &$3 A_1 + D_6 + D_8$	& $\left[{\frac{1}{2}}, {\frac{1}{2}}, {\frac{1}{2}}, {\frac{1}{2}}, 0, {\frac{1}{2}}, 0, 0, {\frac{1}{2}}, 0, 0, 0, 0, 0, 0, {\frac{1}{2}}, {\frac{1}{2}}\right],
\left[0, 0, {\frac{1}{2}}, {\frac{1}{2}}, 0, {\frac{1}{2}}, 0, {\frac{1}{2}}, 0, {\frac{1}{2}}, 0, {\frac{1}{2}}, 0, {\frac{1}{2}}, 0, 0, {\frac{1}{2}}\right]$ \\[0.2em]
$(19,3,1)$ &$A_1 + D_4 + 2 D_6$	& $\left[{\frac{1}{2}}, {\frac{1}{2}}, 0, {\frac{1}{2}}, 0, {\frac{1}{2}}, 0, {\frac{1}{2}}, 0, 0, {\frac{1}{2}}, 0, 0, 0, 0, {\frac{1}{2}}, {\frac{1}{2}}\right],
\left[0, {\frac{1}{2}}, 0, 0, {\frac{1}{2}}, {\frac{1}{2}}, 0, {\frac{1}{2}}, 0, {\frac{1}{2}}, 0, {\frac{1}{2}}, 0, {\frac{1}{2}}, 0, {\frac{1}{2}}, 0\right]$ \\[0.2em]
$(20,2,1)$ &$2A_1 + 2 D_8$		& $\left[{\frac{1}{2}}, {\frac{1}{2}}, {\frac{1}{2}}, 0, {\frac{1}{2}}, 0, {\frac{1}{2}}, 0, {\frac{1}{2}}, 0, 0, 0, 0, 0, 0, 0, {\frac{1}{2}}, {\frac{1}{2}}\right],
\left[0, 0, {\frac{1}{2}}, 0, {\frac{1}{2}}, 0, {\frac{1}{2}}, 0, 0, {\frac{1}{2}}, {\frac{1}{2}}, 0, {\frac{1}{2}}, 0, {\frac{1}{2}}, 0, 0, {\frac{1}{2}}\right]$ \\[0.2em]
$(20,2,1)$ &$3 D_6$				& $\left[0, 0, 0, 0, {\frac{1}{2}}, {\frac{1}{2}}, {\frac{1}{2}}, 0, {\frac{1}{2}}, 0, 0, {\frac{1}{2}}, {\frac{1}{2}}, 0, {\frac{1}{2}}, 0, 0, {\frac{1}{2}}\right],
\left[{\frac{1}{2}}, 0, {\frac{1}{2}}, 0, {\frac{1}{2}}, 0, 0, 0, 0, 0, {\frac{1}{2}}, {\frac{1}{2}}, {\frac{1}{2}}, 0, {\frac{1}{2}}, 0, {\frac{1}{2}}, 0\right]$\\[0.2em]
\hline
\end{tabular}}
\captionsetup{justification=centering}
\caption{Vectors generating a $\mathbb{Z}_2 \times \mathbb{Z}_2$-overlattice of $K^{\text{root}}$}
\label{tab:max_vects_Z2xZ2_infinite}
\end{table}
\bibliographystyle{amsplain}
\bibliography{references}{}
\end{document}